\documentclass[11pt,a4paper]{article}
\usepackage{amsmath,amssymb,latexsym,amsthm}
\usepackage[english]{babel}
\usepackage[latin2]{inputenc}

\begin{document}

\title{Affine reductive spaces of small dimension  and 
left A-loops}
\author{\'Agota Figula} 
\date{}
\maketitle
\newtheorem{Theo}{Theorem}
\newtheorem{Lemma}[Theo]{Lemma}
\newtheorem{Prop}[Theo]{Proposition}
\newcommand{\Rem}{\noindent{\em Remark. }}



\noindent
{\footnotesize {{\em Thanks: } This paper was supported by DAAD}}

\begin{abstract} In this paper we determine the at least $4$-dimensional 
affine reductive homogeneous 
manifolds 
for an at most $9$-dimensional simple Lie group  
or an at most $6$-dimensional semi-simple Lie group. 
Those reductive spaces among them which admit a sharply transitive 
differentiable section yield local almost differentiable left A-loops. 
Using this we classify all global almost differentiable left A-loops $L$ 
having either a $6$-dimensional semi-simple Lie group or the group 
$SL_3(\mathbb R)$ as the group topologically generated by their left 
translations. Moreover,  we determine 
all at most $5$-dimensional left A-loops $L$ with 
 $PSU_3(\mathbb C,1)$ as the group 
topologically generated by their left translations. 
\end{abstract}

\section{ Introduction}

The affine reductive spaces are 
essential objects of differential geometry (cf. \cite{helgason}, 
\cite{oneill}, \cite{kobayashi}). They are homogeneous manifolds $G/H$ 
such that there exists an 
$Ad(H)$-invariant 
subspace ${\bf m}$ of the Lie algebra ${\bf g}$ of $G$ that is complementary 
to the subalgebra ${\bf h}$ in  ${\bf g}$. 

\smallskip
\noindent
The explicite knowledge of affine reductive spaces plays an important 
role in many investigations (cf. \cite{wolf}, \cite{atri}, \cite{kowalski}).
This paper is an application to  differentiable loops since the affine 
reductive spaces are the key for the classification of almost 
differentiable left A-loops $L$; these are loops in which any mapping 
$x \mapsto [(a b)^{-1}(a (b x))], a, b \in L$ is an automorphism of  $L$. 
The relations 
between them  and reductive homogeneous 
spaces  are explicitly discussed  
in \cite{kikkawa1}, \cite{kikkawa2} and \cite{loops}.  

\smallskip
\noindent
Using the fact that the groups topologically generated by the left 
translations of almost differentiable left A-loops $L$ are Lie groups 
(cf. \cite{quasigroups}),    
we treat  $L$ as images of global 
differentiable 
sections $\sigma :G/H \to G$, where $G$ is a connected Lie 
group, $H$ is a 
closed subgroup containing no non-trivial normal subgroup of $G$ 
such that  the subset $\sigma (G/H)$ is invariant 
under the conjugation with the elements of $H$. 
Since the tangent space $T_1(\sigma (G/H))$ is a complementary reductive 
subspace to the Lie algebra ${\bf h}$ of $H$ the  affine 
reductive spaces are  crucial for the classification of 
almost differentiable 
left A-loops. 

\smallskip
\noindent
In contrast to the compact connected Lie groups in which for any 
connected closed subgroup there is an reductive complement 
(cf. \cite{kobayashi}, p. 199), for non-compact Lie groups the situation is 
complicated already if they have small dimension. This is documented by 
 Section 3 and Proposition \ref{reductiveslg2}, where we determine  
all  at least $4$-dimensional affine 
reductive homogeneous spaces $({\bf g}, {\bf h}, {\bf m})$, such that 
${\bf g}$ is either an at 
most $9$-dimensional simple Lie algebra or it is isomorphic to 
$\mathfrak{sl_2}(\mathbb R) \oplus {\bf g}_2$, where ${\bf g}_2$ 
is a $3$-dimensional simple Lie algebra. 

\smallskip
\noindent
The exponential images $\exp {\bf m}$  of  reductive 
complements ${\bf m}$ of the triples  $({\bf g}, {\bf h}, {\bf m})$ 
obtained in Section 3 and in Proposition \ref{reductiveslg2}  yield local left A-loops. 
In Section 4 and Proposition 21 we discuss which of these left A-loops can 
be extended to global ones. They are precisely those exponential images 
$\exp {\bf m}$ which form 
systems of representatives for the cosets $\{xH\ | \ x \in G \}$ in $G$ 
and   do  not contain any element conjugate to an element 
of $H$.  
  
\smallskip
\noindent
Since differentiable Bruck loops have 
realizations on differentiable 
affine symmetric spaces $G/H$, where $H$ is the set of fixed elements of an 
involutory automorphism of $G$ and 
$\sigma (G/H)$ is the exponential image of the $(-1)$-eigenspace of the 
corresponding   automorphism of the Lie algebra ${\bf g}$ of $G$,  the 
class of differentiable 
Bruck loops form a proper 
subclass of almost differentiable left A-loops. An important subclass of 
Bruck loops are the 
Bruck 
loops of hyperbolic type which correspond to Lie groups $G$ and 
involutions $\tau $ fixing 
elementwise a maximal compact subgroup of $G$ (cf. \cite{freudenthal}, 64.9, 
64.10). Almost differentiable left A-loops $L$ having dimension at most 
$3$ and semi-simple Lie groups as the groups  
topologically generated by their  left translations are  classified 
in \cite{loops}, Section 27 and in 
\cite{figula3}. 
Hence in the following main result of this paper only at most 
$4$-dimensional almost differentiable left A-loops occur. 

\medskip
\noindent
{\it {\bf Theorem} Let $L$ be a connected almost differentiable 
left A-loop such that $\hbox{dim} \ L \ge 4$ 
and the group topologically generated by the left translations of $L$ 
is semi-simple. 

\smallskip
\noindent
If $\hbox{dim} \ G=6$ then $G$ is isomorphic to 
$PSL_2(\mathbb R) \times G_2$, where $G_2$ is 
either  $PSL_2(\mathbb R)$ or  $SO_3(\mathbb R)$ and the loop $L$ 
is either a Scheerer extension of 
$G_2$  by the hyperbolic plane loop $\mathbb H_2$ 
(cf. \cite{loops}, Section 22) or the direct product 
$\mathbb H_2 \times \mathbb H_2$. 

\smallskip
\noindent
If the group $G$ is simple and $7 \le \hbox{dim} \ G \le 9$  then $G$ 
is isomorphic either to 
$SL_3(\mathbb R)$ or to 
$PSU_3(\mathbb C,1)$.  In the first case  $L$ is  the 
$5$-dimensional Bruck loop of 
hyperbolic type having the group $SO_3(\mathbb R)$ as the stabilizer 
of $e \in L$ 
(cf. \cite{figula2}, p. 12). In the case 
$G \cong PSU_3(\mathbb C,1)$  every loop $L$ with $\hbox{dim}\ L<6$ 
is  the complex hyperbolic 
plane loop $L_0$ having the group 
$Spin_3 \times SO_2(\mathbb R) /\langle (-1,-1) \rangle $ as the 
stabilizer of $e \in L_0$ (cf. \cite{figula2}, p. 9). }

\section{ Some basic notions}  

\noindent
A binary system  $(L, \cdot )$ is called a loop if there exists an element 
$e \in L$ such that $x=e \cdot x=x \cdot e$ holds for all $x \in L$ and the 
equations 
$a \cdot y=b$ and $x \cdot a=b$ have precisely one solution which we denote 
by $y=a \backslash b$ and $x=b/a$.  Let 
$(L_1, \cdot )$ and $(L_2, \ast )$ be two loops. The set  
$L=L_1 \times L_2= \{ (a,b) \ |\ a \in L_1, b \in L_2 \}$ with the 
componentwise multiplication  is again a loop, which is called the direct 
product of $L_1$ and 
$L_2$, and the loops $(L_1, \cdot )$,  $(L_2, \ast )$ are subloops of $L$.   

\smallskip
\noindent
A loop is called a left A-loop if each mapping 
$\lambda _{x,y}=\lambda _{xy}^{-1} \lambda _x \lambda _y:L \to L$ is an 
automorphism of $L$.

\smallskip
\noindent
Let $G$ be the group generated by the left 
translations of $L$ and let $H$ be the stabilizer of $e \in L$ in the group 
$G$. 
The left translations of $L$ form a subset of $G$ acting on the cosets 
$\{x H; x \in G\}$ such that for any given cosets $aH$ and $bH$ there exists 
precisely one left translation $ \lambda _z$ with $ \lambda _z a H=b H$. 

\smallskip
\noindent
Conversely,  let $G$ be a group, H be a subgroup containing no normal 
non-trivial subgroup of $G$ and $\sigma : G/H \to G$ be a section with 
$\sigma (H)=1 \in G$ 
such that the set $\sigma (G/H)$ of 
representatives for the left cosets $\{ x H, x \in G \}$  
 acts sharply transitively on the space $G/H$ of  $\{x H, x \in G\}$ (cf. \cite{loops}, 
p. 18). Such a section we call a sharply transitive section. 
Then the multiplication  defined by 
$x H \ast y H=\sigma (x H) y H$ on the factor space $G/H$ or by $x \ast y=\sigma(xyH)$ on 
$\sigma (G/H)$ yields a loop $L(\sigma )$. The group $G$ is isomorphic to the group 
generated by the left translations of $L(\sigma )$. 

\smallskip
\noindent
If $G$ is a Lie group and $\sigma $ is a differentiable section 
satisfying the above conditions then the loop $L(\sigma )$ is almost 
differentiable. 
This loop is a left A-loop if and only if the subset 
$\sigma (G/H)$ is invariant under the conjugation with the elements of $H$.
Moreover the manifold $L$ is parallelizable since the set of the left 
translations is 
sharply transitive. 

\smallskip
\noindent
Let $L_1$ be a loop defined on the factor space $G_1/H_1$ with respect to a 
section 
$\sigma _1:G_1/H_1 \to G_1$ the image of which is the set $M_1 \subset G_1$. 
Let $G_2$ be a  group, let $\varphi :H_1 \to G_2$ be a homomorphism and 
$(H_1, \varphi (H_1))=\{(x, \varphi (x)); x \in H_1\}$. A 
loop $L$ is called a Scheerer extension of 
$G_2$ by $L_1$ if the loop $L$ is defined on the factor space 
 $(G_1 \times G_2)/(H_1, \varphi (H_1))$   
with respect to the section 
$\sigma:(G_1 \times G_2)/(H_1, \varphi (H_1)) \to G_1 \times G_2$ the image of which is 
the set $M_1 \times G_2$. 

\smallskip
\noindent
If $L$ is a connected almost differentiable left A-loop, then the group $G$ 
topologically generated by the left translations of $L$ within the group of 
autohomeomorphisms is
 a connected Lie group (cf. \cite{quasigroups}; \cite{loops}, Proposition 5.20. p. 75), 
and we may describe $L$ by a differentiable section. 

\smallskip
\noindent
Let $L$ be a connected almost differentiable left A-loop. Let $G$ be the Lie 
group topologically generated by the left translations of $L$, and let 
$(\bf{g, [.,.]}) $ be the Lie algebra of 
$G$. Denote by $\bf{h}$ the Lie algebra of the stabilizer $H$ of 
the identity 
$e \in L$ in 
$G$ and by ${\bf m}=T_1 \sigma (G/H)$ the tangent space at $1 \in G$ of 
the image of the section $\sigma :G/H \to G$ corresponding to 
 $L$. Then ${\bf m}$ generates ${\bf g}$ and the homogeneous space $G/H$ is 
reductive, i.e.
we have $\bf{g}=\bf{m} \oplus \bf{h}$ and 
$[{\bf h}, {\bf m}] \subseteq {\bf m}$. 
(cf. \cite{loops}, Proposition 5.20. p. 75)  
If $[{\bf m}, {\bf m}] \subseteq {\bf h}$ then the factor space $G/H$ is 
an affine symmetric 
space (\cite{loos})  and the corresponding loop $L$ is called a Bruck loop.

\medskip
\noindent
In our computation we often use the following facts about the Lie algebras 
$\mathfrak{sl_2}(\mathbb R)$ and $\mathfrak{so_3}(\mathbb R)$. 

\medskip
\noindent
As a real basis of $\mathfrak{sl_2}(\mathbb R)$ we choose the following 

\medskip
\noindent
$(\ast )$ \centerline{$e_1=\left (\begin{array}{cc} 
1 & 0 \\ 
0 & -1 \end{array} \right )$, $e_2=\left (\begin{array}{cc} 
0 & 1 \\
1 & 0 \end{array} \right )$, $e_3=\left (\begin{array}{cc}
0 & 1 \\
-1 & 0 \end{array} \right )$} 
(cf. \cite{hilgert}, pp. 19-20).  

\medskip
\noindent
With respect to this basis  the Lie algebra multiplication is given by: 

\medskip
\noindent
\centerline{$[e_1,e_2]=2 e_3$, \ $[e_1,e_3]=2 e_2$, \ $[e_3,e_2]=2 e_1$. }

\medskip
\noindent
{\bf 1.1} An element $X=\lambda _1 e_1+ \lambda _2 e_2+ \lambda _3 e_3 
\in \mathfrak{sl_2}(\mathbb R)$ is elliptic, parabolic or 
hyperbolic according whether  

\medskip
\noindent
\centerline{$k(X)=k(X,X)=\lambda _1^2+\lambda _2^2-\lambda _3^2$\ \  is smaller, 
equal, or greater $0$. }

\medskip
\noindent
The basis elements $e_1,e_2$ are hyperbolic, $e_3$ 
is elliptic and the elements $e_2+e_3$, $e_1+e_3$ are both parabolic. 
All elliptic elements, all hyperbolic elements as well as all parabolic 
elements of $\mathfrak{sl_2}(\mathbb R)$ are  conjugate in this order to 
$e_3$, to $e_1$ respectively  to $e_2+e_3$ (cf. \cite{hilgert}, p. 23). 
There are $3$ conjugacy classes of the one dimensional subgroups of $PSL_2(\mathbb R)$. As  
representatives of these classes  we can  choose 
$ \exp e_3$, $\exp e_1$, $ \exp e_2+e_3$. 
There is precisely one conjugacy class $\mathcal{C}$ of the two dimensional 
subgroups of $PSL_2(\mathbb R)$, as a representative of $\mathcal{C}$ we 
choose 

\medskip
\noindent
\centerline{${\mathcal L}_2= \left \{ \left ( \begin{array}{cc}
a & b \\
0 & a^{-1} \end{array} \right );\ a>0, b \in \mathbb R \right \}$. }
 
\medskip
\noindent
The Lie algebra of ${\mathcal L}_2$ is generated by the elements $e_1$, 
$e_2+e_3$. 

\medskip
\noindent
According to \cite{hilgert} for the exponential function  
$\exp :sl_2(\mathbb R) \to SL_2(\mathbb R)$ we have 
\[ \exp\ X=C(k(X))\ I+S(k(X))\ X. \]   
Here is 
\begin{equation} C(x)= \begin{array}{c}
\cosh \sqrt{x} \quad \hbox{for} \ \ \ 0 \le x, \\
\cos \sqrt{-x} \quad \hbox{for} \ \ \ 0 > x,  \end{array}  \quad 
\sqrt{ | x |} \ S(x)=  \begin{array}{c}
\sinh \sqrt{x} \quad \hbox{for} \ \ \ 0 \le x, \\
\sin \sqrt{-x} \quad \hbox{for} \ \ \ 0 > x.  \end{array} \nonumber \end{equation}  

\medskip
\noindent
{\bf 1.2} As a real basis of the Lie algebra 
$\mathfrak{so_3}(\mathbb R) \cong \mathfrak{su_2}(\mathbb C)$ we can choose 
the basis elements $\{\hbox{i} e_1, \hbox{i} e_2, e_3\}$, where 
$\hbox{i}^2=-1$. 
Every element of $\mathfrak{so_3}(\mathbb R)$ is conjugate to $e_3$. 
\newline
If $X \in \mathfrak{so_3}(\mathbb R)$ has the decomposition 

\medskip
\noindent
\centerline{$X=\lambda _1 \hbox{i} e_1+\lambda _2 \hbox{i} e_2+\lambda _3 e_3$}

\medskip
\noindent
then the normalized real 
Cartan-Killing form $k: \mathfrak{so_3}(\mathbb R) \times \mathfrak{so_3}(\mathbb R) \to 
\mathbb R$; $k(X,Y)=\frac{1}{8} \hbox{trace} (\hbox{ad} X \ \hbox{ad}  Y)$ 
satisfies 

\medskip
\noindent
\centerline{$k(X)=k(X,X)=-\lambda _1^2- \lambda _2^2- \lambda _3^2$.} 

\medskip
\noindent
For the exponential function $\exp : \mathfrak{su_2}(\mathbb C) \to SU_2(\mathbb C)$ one has 
\[ \exp\ X=C(k(X))\ I+S(k(X))\ X, \]   
where $C(x)= \cosh (\sqrt{-x} \hbox{i})$ and 
$S(x)=\displaystyle \frac{\sinh (\sqrt{-x} \hbox{i})}{\sqrt{-x} \hbox{i}}$.

\begin{Prop} \label{kompakt} There is no connected almost differentiable  
left A-loop $L$ such that the 
group $G$ topologically generated by 
its left translations is a  compact quasi-simple Lie group $G$ with 
$\hbox{dim} \ G \le 9$. 
\end{Prop} 
\begin{proof} If $G$ is a  quasi-simple Lie group then it admits a continuous section if and only if $G$ is 
locally isomorphic to  $SO_8(\mathbb R)$ (cf. \cite{scheerer}, pp. 149-150). 
\end{proof}

\medskip
\noindent
An important tool to exclude certain stabilizers $H$ is the fundamental 
group $\pi _1$ of a connected topological 
space. This shows the following lemma which is proved in \cite{figula2}, p. 6.  

\begin{Lemma} \label{fundamenta} Denote by $G$  a connected Lie group and by $H$  a connected  
subgroup of $G$. Let 
$\sigma : G/H \to G$ be a global section. Then 
$\pi_1(K) \cong \pi_1(\sigma (G/H)) \times \pi_1(K_1)$, where 
$K$ respectively $K_1$ is  a maximal compact subgroup of $G$ respectively 
$H$. 
\end{Lemma}

\medskip
\noindent
>From \cite{figula3} we use Lemma 2, which reads as follows.

\begin{Lemma} \label{conjugate} Let $L$ be an almost  differentiable loop and denote by 
${\bf m}$ the 
tangent space $T_1 \sigma(G/H)$, where 
$\sigma : G/H \to G$ is the section corresponding 
to $L$. Then ${\bf m}$ does not contain any element of $Ad_g {\bf h}$ for 
some $g \in G$.  Moreover, every element 
of $G$ can be uniquely written as a product of an 
element of $\sigma(G/H)$ with an element of $H$. 
\end{Lemma}

\section{ Affine reductive spaces of small dimension}

\noindent 
In this section we determine all affine reductive homogeneous spaces 
$({\bf g}, {\bf h}, {\bf m})$, where ${\bf g}$ is a simple non-compact 
Lie algebra of dimension at most 9 and ${\bf h}$ is a subalgebra of 
${\bf g}$ such that $\hbox{dim}\  {\bf g}- \hbox{dim}\  {\bf h} >3$. 

\medskip
\noindent
First we deal with the Lie algebra ${\bf g}=\mathfrak{sl_2}(\mathbb C)$. 
A real basis of  ${\bf g}$ is given by  
$\{ e_1, e_2, e_3, \hbox{i} e_1, \hbox{i} e_2, \hbox{i} e_3 \}$, where
$\{e_1, e_2, e_3 \}$ is the basis of $\mathfrak{sl_2}(\mathbb R)$ described 
by $(\ast )$. 
\newline
\noindent
Using the classification of Lie (see Theorem 15 in 
\cite{lie}, p. 129)  we obtain that  every $2$-dimensional  Lie algebra 
${\bf h}$ of  ${\bf g}$ has (up to conjugation) one of the following shapes:  

\medskip
\noindent 
\centerline{${\bf h}_1=\langle e_1, e_2+e_3 \rangle $, \ \  
${\bf h}_2=\langle \hbox{i} (e_2+e_3), e_2+e_3 \rangle $, \ \ 
${\bf h}_3=\langle e_3, \hbox{i} e_3 \rangle $, } 

\medskip
\noindent 
and  every $1$-dimensional  Lie algebra ${\bf h}$ of  ${\bf g}$ 
is one of the following:

\medskip
\noindent 
\centerline{  
${\bf h}_4=\langle e_1 \rangle $, \ \   
${\bf h}_5=\langle e_2+e_3 \rangle $, \ \  ${\bf h}_6=\langle e_3 \rangle $.}

\medskip
\noindent 
\begin{Prop} \label{sl2c} The Lie algebra ${\bf g}=\mathfrak{sl_2}(\mathbb C)$ is reductive 
with respect to the following pairs  $({\bf h}, {\bf m})$, where ${\bf h}$ is an at most 
$2$-dimensional subalgebra of ${\bf g}$ and ${\bf m}$ is a complementary subspace to ${\bf h}$ 
 generating  ${\bf g}$ 

\medskip
\noindent 
1) \ ${\bf h}_3=\langle e_3, \hbox{i} e_3 \rangle $, ${\bf m}= \langle e_1, e_2, \hbox{i} e_1, \hbox{i} e_2 \rangle $,  

\medskip
\noindent 
2) \ ${\bf h}_4=\langle e_1 \rangle $, \ ${\bf m}_a=\langle e_2, e_3, \hbox{i} e_1+a e_1, \hbox{i} e_2, \hbox{i} e_3 \rangle$, \ where $a \in \mathbb R$,  

\medskip
\noindent 
3) \ ${\bf h}_6=\langle e_3 \rangle $, \ ${\bf m}_b=\langle e_1, e_2, \hbox{i} e_1,  
\hbox{i} e_2, \hbox{i} e_3+b e_3 \rangle$, \ where $b \in \mathbb R$. 
\end{Prop}
\begin{proof} The basis elements of an arbitrary complement ${\bf m}_1$ to ${\bf h}_1$ in ${\bf g}$ are 

\medskip
\noindent 
\centerline{$X_1=e_2+a_1 e_1+b_1(e_2+e_3)$, \ \ $X_2=\hbox{i} e_1+a_2 e_1+b_2(e_2+e_3)$, }

\medskip
\noindent 
\centerline{$X_3=\hbox{i} e_2+a_3 e_1+b_3(e_2+e_3)$, \ \ $X_4=\hbox{i} e_3+a_4 e_1+b_4(e_2+e_3)$, } 

\medskip
\noindent 
where $a_j$, $b_j$, $j=1,2,3,4$ are real parameters. 
\newline
An arbitrary complement ${\bf m}_2$ to ${\bf h}_2$ in  ${\bf g}$ has 
as generators 

\medskip
\noindent 
\centerline{ $Y_1=e_1+a_1(e_2+e_3)+b_1 \hbox{i} (e_2+e_3)$, \ \ $Y_2=e_2+a_2(e_2+e_3)+b_2 \hbox{i} (e_2+e_3)$, }

\medskip
\noindent
\centerline{ $Y_3=\hbox{i} e_1+a_3(e_2+e_3)+b_3 \hbox{i} (e_2+e_3)$, \ \ $Y_4=\hbox{i} e_2+a_4(e_2+e_3)+b_4 \hbox{i} (e_2+e_3)$, } 

\medskip
\noindent 
where $a_j, b_j \in \mathbb R$, $j=1,2,3,4$.
\newline
We can choose as basis elements of an arbitrary complement ${\bf m}_3$ to ${\bf h}_3$ the following:

\medskip
\noindent 
\centerline{$Z_1=e_1+a_1 e_3+b_1 \hbox{i} e_3$, \ \ $Z_2=e_2+a_2 e_3+b_2 \hbox{i} e_3$, }

\medskip
\noindent 
\centerline{$Z_3=\hbox{i} e_1+a_3 e_3+b_3 \hbox{i} e_3$, \ \ $Z_4=\hbox{i} e_2+a_4 e_3+b_4 \hbox{i} e_3$, }
 
\medskip
\noindent 
where $a_j, b_j \in \mathbb R$, $j=1,2,3,4$ are real numbers. 
\newline
An arbitrary complement ${\bf m}_4$ to ${\bf h}_4$ in ${\bf g}$ 
has as basis elements 

\medskip
\noindent 
\centerline{ $W_1=e_2+a_1 e_1$, \ \ $W_2=e_3+a_2 e_1$, \ \ $W_3=\hbox{i} e_1+a_3 e_1$, }

\medskip
\noindent 
\centerline{$W_4=\hbox{i} e_2+a_4 e_1$, \ \ $W_5=\hbox{i} e_3+a_5 e_1$  } 

\medskip
\noindent 
with the real parameters $a_j$, $j=1,2,3,4,5$. 
\newline 
The generators of an arbitrary complement ${\bf m}_5$ to ${\bf h}_5$ 
in ${\bf g}$ are 

\medskip
\noindent 
\centerline{ $V_1=e_1+a_1(e_2+e_3)$, \ \ $V_2=e_2+a_2(e_2+e_3)$, \ \ 
$V_3=\hbox{i} e_1+a_3(e_2+e_3)$, }

\medskip
\noindent 
\centerline{ $V_4=\hbox{i} e_2+a_4(e_2+e_3)$, \ \ 
$V_5=\hbox{i} e_3+a_5(e_2+e_3)$, } 

\medskip
\noindent 
where $a_j$, $j=1,2,3,4,5$ are real parameters.  
\newline
An arbitrary complement ${\bf m}_6$ to ${\bf h}_6$ in  ${\bf g}$ has as 
generators 

\medskip
\noindent 
\centerline{ $U_1=e_1+a_1 e_3$, \ \ $U_2=e_2+a_2 e_3$, \ \ 
$U_3=\hbox{i} e_1+a_3 e_3$, }

\medskip
\noindent 
\centerline{ $U_4=\hbox{i} e_2+a_4 e_3$, \ \ $U_5=\hbox{i} e_3+a_5 e_3$ } 

\medskip
\noindent 
with $a_1, a_2, a_3, a_4, a_5 \in \mathbb R$.
\newline
Using the relation $[{\bf h}_i,{\bf m}_i] \subseteq {\bf m}_i$, 
$i=1, \cdots ,6$, we obtain the 
contradictions that $[e_2+e_3, X_1]=2 e_1-2 a_1(e_2+ e_3) \in {\bf h}_1$ 
and 
$[e_2+e_3,Y_1]=[e_2+e_3,V_1]=-2(e_2+e_3) \in {\bf h}_2 \cap {\bf h}_5$ 
and the assertion follows. 
\end{proof}

\bigskip
\noindent
Now we consider the Lie algebra $\bf{g}=\mathfrak{sl_3}(\mathbb R)$.  
It is isomorphic to the Lie 
algebra of  matrices
\[(\lambda _1 e_1+\lambda _2 e_2+\lambda _3 e_3+\lambda _4 e_4+\lambda _5 e_5+\lambda _6 e_6+\lambda _7 e_7+\lambda _8 e_8) \mapsto \]
\[ \left ( \begin{array}{ccc}
- \lambda _5-\lambda _8 & \lambda _1 & \lambda_2 \\
 \lambda _3 &  \lambda _5 & \lambda _6 \\
 \lambda _4 & \lambda _7 & \lambda _8 \end{array} \right ); \lambda_i \in 
\mathbb R, i=1,\cdots ,8. \] 
In this representation  the Lie multiplication  of ${\bf g}$ is given by 
\[ [e_1,e_2]=[e_1,e_7]=[e_2,e_6]=[e_3,e_4]=[e_3,e_6]=[e_4,e_7]=[e_5,e_8]=0, \]
\[ [e_1,e_6]=[e_2,e_5]=\frac{1}{2} [e_2,e_8]=e_2,\  [e_1,e_8]=[e_2,e_7]=\frac{1}{2} [e_1,e_5]=e_1, \]
\[ [e_4,e_6]=[e_3,e_8]=\frac{1}{2} [e_3,e_5]=-e_3,\  [e_3,e_7]=[e_4,e_5]=\frac{1}{2} [e_4,e_8]=-e_4, \]
\[ [e_6,e_8]=[e_5,e_6]=[e_3,e_2]=e_6,\  [e_1,e_4]=[e_5,e_7]=[e_7,e_8]=-e_7,  \]
\[ [e_1,e_3]=- e_5,\  [e_2,e_4]=-e_8,\  [e_6,e_7]=e_5-e_8. \] 
Now using the classification  of 
Lie, who has determined all subalgebras of 
$\mathfrak{sl_3}(\mathbb R)$  (cf. \cite{lie}, pp. 288-289 and \cite{lie2}, 
p. 384) we obtain that every $4$-dimensional Lie algebra ${\bf h}$ of  ${\bf g}$  has 
(up to conjugation) one of the following forms: 

\medskip
\noindent
${\bf h}_1=\langle e_1, e_2, e_6, e_5+c e_8 \rangle $, \   \  
${\bf h}_2=\langle e_3, e_5, e_6, e_8 \rangle $, \ \ 
${\bf h}_3=\langle e_1, e_2, e_6, e_8 \rangle $, 
\newline
${\bf h}_4=\langle e_2, e_5, e_6, e_8  \rangle $, \ \ 
${\bf h}_5 \cong \mathfrak{gl_2}(\mathbb R)= 
\langle e_5, e_6, e_7, e_8  \rangle $,
where $c \in \mathbb R$. 

\medskip
\noindent
The $3$-dimensional subalgebras ${\bf h}$ of  ${\bf g}$ (up to conjugation) 
are the following:

\medskip
\noindent
${\bf h}_6 \cong \mathfrak{so_3}(\mathbb R)= \langle e_1-e_3, e_2-e_4, e_7-e_6 \rangle $, \ 
${\bf h}_7 \cong \mathfrak{sl_2}(\mathbb R)= \langle e_1+e_3, e_2+e_4, e_6-e_7 \rangle $,
\newline 
${\bf h}_8 \cong \mathfrak{sl_2}(\mathbb R)= \langle e_5-e_8, e_6, e_7 \rangle $,
${\bf h}_9= \langle a(e_5+e_8)+e_6-e_7, e_1, e_2 \rangle $, $a \ge 0$, \ 
\newline
${\bf h}_{10}=  \langle e_5- e_8, e_2+ e_3, e_6 \rangle $,
${\bf h}_{11}=  \langle  e_3, e_6, e_8+e_2  \rangle $, 
${\bf h}_{12}= \langle  e_2, e_6,  e_5+ e_8- e_3 \rangle $, 
\newline
${\bf h}_{13}= \langle e_1, e_2, e_6 \rangle $,
${\bf h}_{14}= \langle  e_5, e_8,  e_6 \rangle $, 
${\bf h}_{15}=  \langle  e_2, e_5+ e_8,  e_6  \rangle $, 
${\bf h}_{16}= \langle e_3, e_6, e_8 \rangle $,   
\newline
${\bf h}_{17}= \langle e_2, e_6, (b-1)e_5+b e_8, 
\rangle $, $b \in \mathbb R$,
${\bf h}_{18}= \langle e_3, e_6,  e_5+ c e_8 \rangle $, $c \in \mathbb R$. 

\medskip
\noindent
The $2$-dimensional subalgebras ${\bf h}$ of  ${\bf g}$ are given 
(up to conjugation) by 

\medskip
\noindent
${\bf h}_{19}= \langle e_6,  e_2+ e_3 \rangle $, \ \ ${\bf h}_{20}= 
\langle e_6,  e_2+ e_8 \rangle $,\  \ ${\bf h}_{21}= \langle e_3,  
e_6+ e_5 \rangle $, 
\newline
${\bf h}_{22}= \langle e_3,  e_5+ a e_8 \rangle $, \ $a \in \mathbb R \backslash \{0,1 \}$, \ \  
${\bf h}_{23}= \langle e_5,  e_6 \rangle $, \ \ 
${\bf h}_{24}= \langle e_2,  e_6 \rangle $, 
\newline
${\bf h}_{25}= \langle e_6,  e_3 \rangle $, \ \ ${\bf h}_{26}= \langle e_5, 
e_8 \rangle $, \ \ 
${\bf h}_{27}= \langle e_6,  e_5+ e_8 \rangle $, \ \ 
${\bf h}_{28}= \langle e_6,  e_8 \rangle $, 
\newline
${\bf h}_{29}= \langle e_5-e_8, e_2+ e_3 \rangle $,\ \ 
${\bf h}_{30}= \langle e_5+e_8, e_6- e_7 \rangle $. 

\medskip
\noindent
Moreover, every $1$-dimensional subalgebra ${\bf h}$ of  ${\bf g}$ has one  
of the following shapes:

\medskip
\noindent
${\bf h}_{31}= \langle  e_5+ a e_8 \rangle $, \ $a \in \mathbb R \backslash \{0 \}$, \ \ 
${\bf h}_{32}= \langle e_2+e_8 \rangle $, \ \ 
${\bf h}_{33}= \langle e_2+e_3 \rangle $, 
\newline
${\bf h}_{34}= \langle e_6 \rangle $, \ \  
${\bf h}_{35}= \langle e_6-e_7+b(e_5+ e_8) \rangle $, $b \ge 0$.

\begin{Prop} \label{sl3R4dim} The Lie algebra 
${\bf g}=\mathfrak{sl_3}(\mathbb R)$ is reductive 
with respect to a $4$-dimensional subalgebra ${\bf h}$ of ${\bf g}$ and 
a complementary subspace ${\bf m}$ generating ${\bf g}$ only in the case 
${\bf h}_5 \cong \mathfrak{gl_2}(\mathbb R)$ and  
${\bf m}_5=\langle e_1, e_2, e_3, e_4 \rangle $. 
\end{Prop}
\begin{proof}  The basis elements of an arbitrary complement ${\bf m}_i$ 
to the subalgebra ${\bf h}_i$ are: 
\newline
For $i=1$

\medskip
\noindent
$e_3+a_1 e_1+a_2 e_2+a_3(e_5+c e_8)+a_4 e_6$, \ $e_4+b_1 e_1+b_2 e_2+b_3(e_5+c e_8)+b_4 e_6$,

\medskip
\noindent
$e_7+c_1 e_1+c_2 e_2+c_3(e_5+c e_8)+c_4 e_6$, \ $e_8+d_1 e_1+d_2 e_2+d_3(e_5+c e_8)+d_4 e_6$,

\medskip
\noindent
for $i=2$ 

\medskip
\noindent
\centerline{
$e_1+a_1 e_3+a_2 e_5+a_3 e_6+a_4 e_8$, \ $e_2+b_1 e_3+b_2 e_5+b_3 e_6+b_4 e_8$, }

\medskip
\noindent
\centerline{
$e_4+c_1 e_3+c_2 e_5+c_3 e_6+c_4 e_8$, \ $e_7+d_1 e_3+d_2 e_5+d_3 e_6+d_4 e_8$,}

\medskip
\noindent
for $i=3$ 

\medskip
\noindent
\centerline{$e_3+a_1 e_1+a_2 e_2+a_3 e_6+a_4 e_8$, \ $e_4+b_1 e_1+b_2 e_2+b_3 e_6+b_4 e_8$, }

\medskip
\noindent
\centerline{$e_5+c_1 e_1+c_2 e_2+c_3 e_6+c_4 e_8$, \ $e_7+d_1 e_1+d_2 e_2+d_3 e_6+d_4 e_8$, }

\medskip
\noindent
for $i=4$ 

\medskip
\noindent
\centerline{$e_1+a_1 e_2+a_2 e_5+a_3 e_6+a_4 e_8$, \ $e_3+b_1 e_2+b_2 e_5+b_3 e_6+b_4 e_8$, }

\medskip
\noindent
\centerline{$e_4+c_1 e_2+c_2 e_5+c_3 e_6+c_4 e_8$, \ $e_7+d_1 e_2+d_2 e_5+d_3 e_6+d_4 e_8$, }

\medskip
\noindent
for $i=5$ 

\medskip
\noindent
\centerline{$e_1+a_1 e_5+a_2 e_6+a_3 e_7+a_4 e_8$, \ $e_2+b_1 e_5+b_2 e_6+b_3 e_7+b_4 e_8$, }

\medskip
\noindent
\centerline{$e_3+c_1 e_5+c_2 e_6+c_3 e_7+c_4 e_8$, \ $e_4+d_1 e_5+d_2 e_6+d_3 e_7+d_4 e_8$, }

\medskip
\noindent
where $a_j, b_j, c_j, d_j$ are real numbers $j=1,2,3,4$. The assertion 
follows now from the relation $[{\bf h},{\bf m}] \subseteq {\bf m}$.  
\end{proof}

\begin{Prop} \label{sl3R3dim} The Lie algebra 
${\bf g}=\mathfrak{sl_3}(\mathbb R)$ is 
reductive with a $3$-dimensional subalgebra ${\bf h}$ and a 
$5$-dimensional complementary subspace ${\bf m}$ generating ${\bf g}$  in 
precisely one of the following cases: 

\medskip
\noindent
1) \  ${\bf h}_6 \cong \mathfrak{so_3}(\mathbb R)$,  
${\bf m}_6= \langle e_5,e_8,e_1+e_3, e_2+e_4, e_7+e_6 \rangle$, 

\medskip
\noindent
2) \ ${\bf h}_7= \langle e_1+e_3, e_2+e_4, 
e_6-e_7 \rangle$, 
${\bf m}_7=\langle e_5,e_8,e_1-e_3, e_2-e_4, e_7+e_6 \rangle$, 

\medskip
\noindent
3) \ ${\bf h}_8= \langle e_5-e_8, e_6, e_7 \rangle$, 
${\bf m}_8=\langle e_1, e_2, e_3, e_4, e_5+e_8 \rangle$. 

\medskip
\noindent
Both Lie algebras ${\bf h}_7$ and ${\bf h}_8$ are isomorphic to 
$\mathfrak{sl_2}(\mathbb R)$. 
\end{Prop} 
\begin{proof}  The generators of an arbitrary complement ${\bf m}_i$ to 
${\bf h}_i$ in ${\bf g}$ are: 
\newline 
For $i=6$ 

\medskip
\noindent
\centerline{$e_3+a_1(e_1-e_3)+a_2(e_2-e_4)+a_3(e_7-e_6)$,}

\medskip
\noindent
\centerline{$e_4+b_1(e_1-e_3)+b_2(e_2-e_4)+b_3(e_7-e_6)$,}

\medskip
\noindent
\centerline{$e_5+c_1(e_1-e_3)+c_2(e_2-e_4)+c_3(e_7-e_6)$,}

\medskip
\noindent
\centerline{$e_6+d_1(e_1-e_3)+d_2(e_2-e_4)+d_3(e_7-e_6)$,}

\medskip
\noindent
\centerline{$e_8+f_1(e_1-e_3)+f_2(e_2-e_4)+f_3(e_7-e_6)$,}

\medskip
\noindent
for $i=7$ 

\medskip
\noindent
\centerline{$e_3+a_1(e_1+e_3)+a_2(e_2+e_4)+a_3(e_6-e_7)$,}

\medskip
\noindent
\centerline{$e_4+b_1(e_1+e_3)+b_2(e_2+e_4)+b_3(e_6-e_7)$,}

\medskip
\noindent
\centerline{$e_5+c_1(e_1+e_3)+c_2(e_2+e_4)+c_3(e_6-e_7)$,}

\medskip
\noindent
\centerline{$e_6+d_1(e_1+e_3)+d_2(e_2+e_4)+d_3(e_6-e_7)$,}

\medskip
\noindent
\centerline{$e_8+f_1(e_1+e_3)+f_2(e_2+e_4)+f_3(e_6-e_7)$,}

\medskip
\noindent
for $i=8$ 

\medskip
\noindent
\centerline{$e_1+a_1(e_5-e_8)+a_2 e_6+a_3 e_7$, \ $e_2+b_1(e_5-e_8)+b_2 e_6+b_3 e_7$,}

\medskip
\noindent
\centerline{$e_3+c_1(e_5-e_8)+c_2 e_6+c_3 e_7$, \ $e_4+d_1(e_5-e_8)+d_2 e_6+d_3 e_7$,}

\medskip
\noindent
\centerline{$e_5+f_1(e_5-e_8)+f_2 e_6+f_3 e_7$,}

\medskip
\noindent
for $i=9$ 

\medskip
\noindent
\centerline{$e_3+a_1 e_1+a_2 e_2+a_3(e_6-e_7+a(e_5+e_8))$,}

\medskip
\noindent
\centerline{$e_4+b_1 e_1+b_2 e_2+b_3(e_6-e_7+a(e_5+e_8))$,}

\medskip
\noindent
\centerline{$e_5+c_1 e_1+c_2 e_2+c_3(e_6-e_7+a(e_5+e_8))$, }

\medskip
\noindent
\centerline{$e_6+d_1 e_1+d_2 e_2+d_3(e_6-e_7+a(e_5+e_8))$,}

\medskip
\noindent
\centerline{$e_8+f_1 e_1+f_2 e_2+f_3(e_6-e_7+a(e_5+e_8))$,}

\medskip
\noindent
for $i=10$ 

\medskip
\noindent
$e_1+a_1(e_2+e_3)+a_2(e_5-e_8)+a_3 e_6$, \ $e_2+b_1(e_2+e_3)+b_2(e_5-e_8)+b_3 e_6$,

\medskip
\noindent
$e_4+c_1(e_2+e_3)+c_2(e_5-e_8)+c_3 e_6$, \ $e_5+d_1(e_2+e_3)+d_2(e_5-e_8)+d_3 e_6$,

\medskip
\noindent
\centerline{$e_7+f_1(e_2+e_3)+f_2(e_5-e_8)+f_3 e_6$,} 

\medskip
\noindent
for $i=11$ 

\medskip
\noindent
\centerline{$e_1+a_1(e_2+e_8)+a_2 e_3+a_3 e_6$, \ $e_2+b_1(e_2+e_8)+b_2 e_3+b_3 e_6$,}

\medskip
\noindent
\centerline{$e_4+c_1(e_2+e_8)+c_2 e_3+c_3 e_6$, \ $e_5+d_1(e_2+e_8)+d_2 e_3+d_3 e_6$,}

\medskip
\noindent
\centerline{$e_7+f_1(e_2+e_8)+f_2 e_3+f_3 e_6$,} 

\medskip
\noindent
for $i=12$ 

\medskip
\noindent
\centerline{$e_1+a_1 e_2+a_2 e_6+a_3(e_5+e_8-e_3)$, \ $e_3+b_1 e_2+b_2 e_6+b_3(e_5+e_8-e_3)$, }

\medskip
\noindent
\centerline{$e_4+c_1 e_2+c_2 e_6+c_3(e_5+e_8-e_3)$, \ $e_7+d_1 e_2+d_2 e_6+d_3(e_5+e_8-e_3)$, }

\medskip
\noindent
\centerline{$e_8+f_1 e_2+f_2 e_6+f_3(e_5+e_8-e_3)$,} 

\medskip
\noindent
for $i=13$ 

\medskip
\noindent
\centerline{$e_3+a_1 e_1+a_2 e_2+a_3 e_6$, 
\ $e_4+b_1 e_1+b_2 e_2+b_3 e_6$, \ $e_5+c_1 e_1+c_2 e_2+c_3 e_6$,}

\medskip
\noindent
\centerline{$e_7+d_1 e_1+d_2 e_2+d_3 e_6$, \ $e_8+f_1 e_1+f_2 e_2+f_3 e_6$,}

\medskip
\noindent
for $i=14$ 

\medskip
\noindent
\centerline{$e_1+a_1 e_5+a_2 e_6+a_3 e_8$, \ $e_2+b_1 e_5+b_2 e_6+b_3 e_8$,
\ $e_3+c_1 e_5+c_2 e_6+c_3 e_8$, }

\medskip
\noindent
\centerline{$e_4+d_1 e_5+d_2 e_6+d_3 e_8$, \ 
$e_7+f_1 e_5+f_2 e_6+f_3 e_8$, }

\medskip
\noindent
for $i=15$ 

\medskip
\noindent
\centerline{$e_1+a_1 e_2+a_2 (e_5+e_8)+a_3 e_6$, \ $e_3+b_1 e_2+b_2 (e_5+e_8)+b_3 e_6$, }

\medskip
\noindent
\centerline{$e_4+c_1 e_2+c_2 (e_5+e_8)+c_3 e_6$, \ $e_5+d_1 e_2+d_2 (e_5+e_8)+d_3 e_6$, } 

\medskip
\noindent
\centerline{$e_7+f_1 e_2+f_2 (e_5+e_8)+f_3 e_6$, }

\medskip
\noindent
for $i=16$ 

\medskip
\noindent
\centerline{$e_1+a_1 e_3+a_2 e_6+a_3 e_8$, 
\ $e_2+b_1 e_3+b_2 e_6+b_3 e_8$, \ 
$e_4+c_1 e_3+c_2 e_6+c_3 e_8$, }

\medskip
\noindent
\centerline{$e_5+d_1 e_3+d_2 e_6+d_3 e_8$, \ 
$e_7+f_1 e_3+f_2 e_6+f_3 e_8$, }  

\medskip
\noindent
for $i=17$ and $b \neq 0$ 

\medskip
\noindent
$e_1+a_1 e_2+a_2 e_6+a_3 ((b-1) e_5+b e_8)$, \ $e_3+b_1 e_2+b_2 e_6+b_3 ((b-1) e_5+b e_8)$,

\medskip
\noindent
$e_4+c_1 e_2+c_2 e_6+c_3 ((b-1) e_5+b e_8)$, \ $e_5+d_1 e_2+d_2 e_6+d_3 ((b-1) e_5+b e_8)$, 

\medskip
\noindent
\centerline{$e_7+f_1 e_2+f_2 e_6+f_3 ((b-1) e_5+b e_8)$, }  

\medskip
\noindent
for $i=17$ and $b=0$ 

\medskip
\noindent
$e_1+a_1 e_2+a_2 e_6-a_3 e_5$, \ $e_3+b_1 e_2+b_2 e_6-b_3 e_5$, \ 
$e_4+c_1 e_2+c_2 e_6-c_3 e_5$, 

\medskip
\noindent
\centerline{$e_7+d_1 e_2+d_2 e_6-d_3 e_5$, \ 
$e_8+f_1 e_2+f_2 e_6-f_3 e_5$, } 

\medskip
\noindent
for $i=18$ 

\medskip
\noindent
\centerline{$e_1+a_1 e_3+a_2(e_5+c e_8)+a_3 e_6$, \ $e_2+b_1 e_3+b_2(e_5+c e_8)+b_3 e_6$, }

\medskip
\noindent
\centerline{$e_4+c_1 e_3+c_2(e_5+c e_8)+c_3 e_6$, \ $e_7+d_1 e_3+d_2(e_5+c e_8)+d_3 e_6$, }

\medskip
\noindent
\centerline{$e_8+f_1 e_3+f_2(e_5+c e_8)+f_3 e_6$, } 

\medskip
\noindent
where $a_j, b_j, c_j, d_j, f_j \in \mathbb R$, $j=1,2,3$. Using the relation 
$[{\bf h_i},{\bf m_i}] \subseteq {\bf m_i}$, $i=6, \cdots ,18$, we obtain the assertion. 
\end{proof}

\begin{Prop} \label{sl3R2dim} The Lie algebra ${\bf g}=\mathfrak{sl_3}(\mathbb R)$ is 
reductive with respect to a  pair $({\bf h}, {\bf m})$, where ${\bf h}$ 
is a $2$-dimensional subalgebra of ${\bf g}$ and ${\bf m}$ is a  complementary subspace to ${\bf h}$ 
 generating  ${\bf g}$  in exactly one of  the following cases:  

\medskip
\noindent
1) ${\bf h}_{26}=\langle  e_5, e_8 \rangle $  and 
${\bf m}_{26}=\langle  e_1, e_2, e_3, e_4, e_6, e_7 \rangle$.

\medskip
\noindent
2) ${\bf h}_{30}=\langle  e_5+e_8, e_6-e_7 \rangle $  and 
${\bf m}_{30}=\langle  e_1, e_2, e_3, e_4, e_5-e_8, e_6+e_7 \rangle$. 

\end{Prop} 
\begin{proof} An arbitrary complement ${\bf m}_i$ to the subalgebra ${\bf h}_i$, 
$i=19,\cdots ,30,$ in ${\bf g}$ has as generators in the case $i=19$ 

\medskip
\noindent
\centerline{$e_1+b_1 e_6+ c_1(e_2+e_3)$, \  $e_2+b_2 e_6+ c_2(e_2+e_3)$, \  $e_4+b_3 e_6+ c_3(e_2+e_3)$ }

\medskip
\noindent
\centerline{$e_5+b_4 e_6+ c_4(e_2+e_3)$, \  $e_7+b_5 e_6+ c_5(e_2+e_3)$, \  $e_8+b_6 e_6+ c_6(e_2+e_3)$, }

\medskip
\noindent
in the case $i=20$ 

\medskip
\noindent
\centerline{$e_1+b_1 e_6+ c_1(e_2+e_8)$, \  $e_2+b_2 e_6+ c_2(e_2+e_8)$, \  
$e_3+b_3 e_6+ c_3(e_2+e_8)$, }

\medskip
\noindent
\centerline{$e_4+b_4 e_6+ c_4(e_2+e_8)$, \  $e_5+b_5 e_6+ c_5(e_2+e_8)$, \  $e_7+b_6 e_6+ c_6(e_2+e_8)$,}

\medskip
\noindent
in the case $i=21$ 

\medskip
\noindent
\centerline{$e_1+b_1 e_3+ c_1(e_6+e_5)$, \  $e_2+b_2 e_3+ c_2(e_6+e_5)$, \  $e_4+b_3 e_3+ c_3(e_6+e_5)$,}

\medskip
\noindent
\centerline{$e_5+b_4 e_3+ c_4(e_6+e_5)$, \  $e_7+b_5 e_3+ c_5(e_6+e_5)$, \  $e_8+b_6 e_3+ c_6(e_6+e_5)$,}

\medskip
\noindent
in the case $i=22$ 

\medskip
\noindent
\centerline{$e_1+b_1 e_3+ c_1(e_5+a e_8)$, \  $e_2+b_2 e_3+ c_2(e_5+a e_8)$, \ $e_4+b_3 e_3+ c_3(e_5+a e_8)$,}

\medskip
\noindent
\centerline{$e_6+b_4 e_3+ c_4(e_5+a e_8)$, \  $e_7+b_5 e_3+ c_5(e_5+a e_8)$, \  $e_8+b_6 e_3+ c_6(e_5+a e_8)$,}

\medskip
\noindent
in the case $i=23$ 

\medskip
\noindent
\centerline{$e_1+b_1 e_5+ c_1 e_6$, \ $e_2+b_2 e_5+ c_2 e_6$, \ $e_3+b_3 e_5+ c_3 e_6$,}

\medskip
\noindent
\centerline{$e_4+b_4 e_5+ c_4 e_6$, \ $e_7+b_5 e_5+ c_5 e_6$, \  $e_8+b_6 e_5+ c_6 e_6$,}

\medskip
\noindent
in the case $i=24$ 

\medskip
\noindent
\centerline{$e_1+b_1 e_2+ c_1 e_6$, \ $e_3+b_2 e_2+ c_2 e_6$, \ $e_4+b_3 e_2+ c_3 e_6$,}

\medskip
\noindent
\centerline{$e_5+b_4 e_2+ c_4 e_6$, \ $e_7+b_5 e_2+ c_5 e_6$, \  $e_8+b_6 e_2+ c_6 e_6$,}

\medskip
\noindent
in the case $i=25$ 

\medskip
\noindent
\centerline{$e_1+b_1 e_3+ c_1 e_6$, \ $e_2+b_2 e_3+ c_2 e_6$, \ $e_4+b_3 e_3+ c_3 e_6$,}

\medskip
\noindent
\centerline{$e_5+b_4 e_3+ c_4 e_6$, \ $e_7+b_5 e_3+ c_5 e_6$, \  $e_8+b_6 e_3+ c_6 e_6$,}

\medskip
\noindent
in the case $i=26$ 

\medskip
\noindent
\centerline{$e_1+b_1 e_5+ c_1 e_8$, \ $e_2+b_2 e_5+ c_2 e_8$, \ $e_3+b_3 e_5+ c_3 e_8$,}

\medskip
\noindent
\centerline{$e_4+b_4 e_5+ c_4 e_8$, \ $e_6+b_5 e_5+ c_5 e_8$, \  $e_7+b_6 e_5+ c_6 e_8$,}

\medskip
\noindent
in the case $i=27$ 

\medskip
\noindent
\centerline{$e_1+b_1 e_6+ c_1(e_5+e_8)$, \ $e_2+b_2 e_6+ c_2(e_5+e_8)$ \ $e_3+b_3 e_6+ c_3(e_5+e_8)$,}

\medskip
\noindent
\centerline{$e_4+b_4 e_6+ c_4(e_5+e_8)$, \ $e_5+b_5 e_6+ c_5(e_5+e_8)$, \  $e_7+b_6 e_6+ c_6(e_5+e_8)$,}

\medskip
\noindent
in the case $i=28$ 

\medskip
\noindent
\centerline{$e_1+b_1 e_6+ c_1 e_8$, \ $e_2+b_2 e_6+ c_2 e_8$ \ $e_3+b_3 e_6+ c_3 e_8$,}

\medskip
\noindent
\centerline{$e_4+b_4 e_6+ c_4 e_8$, \ $e_5+b_5 e_6+ c_5 e_8$, \  $e_7+b_6 e_6+ c_6 e_8$,}

\medskip
\noindent
in the case $i=29$ 

\medskip
\noindent
\centerline{$e_1+b_1(e_2+e_3)+ c_1(e_5-e_8)$, \ $e_2+b_2(e_2+e_3)+ c_2(e_5-e_8)$, } 

\medskip
\noindent
\centerline{$e_4+b_3(e_2+e_3)+ c_3(e_5-e_8)$, \ $e_5+b_4(e_2+e_3)+ c_4(e_5-e_8)$, }

\medskip
\noindent
\centerline{$e_6+b_5(e_2+e_3)+ c_5(e_5-e_8)$, \ $e_7+b_6(e_2+e_3)+ c_6(e_5-e_8)$, }

\medskip
\noindent
in the case $i=30$ 

\medskip
\noindent
\centerline{$e_1+b_1(e_5+e_8)+ c_1(e_6-e_7)$, \ $e_2+b_2(e_5+e_8)+ c_2(e_6-e_7)$, } 

\medskip
\noindent
\centerline{$e_3+b_3(e_5+e_8)+ c_3(e_6-e_7)$, \ $e_4+b_4(e_5+e_8)+ c_4(e_6-e_7)$, }

\medskip
\noindent
\centerline{$e_5+b_5(e_5+e_8)+ c_5(e_6-e_7)$, \  $e_6+b_6(e_5+e_8)+ c_6(e_6-e_7)$,}

\medskip
\noindent
where $b_j, c_j \in \mathbb R$, $j=1, \cdots , 6$. The relation $[{\bf h_i},{\bf m_i}] 
\subseteq {\bf m_i}$, $i=19, \cdots ,30,$ yields the assertion. 
\end{proof}

\begin{Prop} \label{sl3R1dim} The Lie algebra ${\bf g}=\mathfrak{sl_3}(\mathbb R)$ is 
reductive with a $1$-dimensional subalgebra ${\bf h}$   and a 
$7$-dimensional complementary subspace ${\bf m}$ generating ${\bf g}$ in precisely 
one of the following 
cases:   

\medskip
\noindent
1) ${\bf h}_{31,1}=\langle  e_5+ a e_8 \rangle $, \ $a \in \mathbb R \backslash \{0,1,-\frac{1}{2},-2 \}$ and 
\newline
${\bf m}_{b}=\langle e_1, e_2, e_3, e_4, e_6, e_7, 
e_8+b(e_5+a e_8) \rangle$, $b \in \mathbb R$,  

\medskip
\noindent
2) ${\bf h}_{31,2}=\langle  e_5-2 e_8 \rangle $ and 
\newline
${\bf m}_{b,c,d}=\langle e_6, e_7, e_1+b (e_5-2 e_8), e_3+c (e_5-2 e_8), e_2, e_4, e_8+d(e_5-2 e_8)  \rangle$, $b,c,d \in \mathbb R$,  

\medskip
\noindent
3) ${\bf h}_{31,3}=\langle  e_5-\frac{1}{2} e_8 \rangle $ and 
\newline
${\bf m}_{b,c,d}=\langle e_6, e_7, e_1, e_2+b (e_5-\frac{1}{2} e_8), e_3, e_4+c (e_5-\frac{1}{2} e_8), e_8+d(e_5-\frac{1}{2} e_8)  \rangle$, $b,c,d \in \mathbb R$,

\medskip
\noindent
4) ${\bf h}_{31,4}=\langle  e_5+ e_8 \rangle $ and 
\newline
${\bf m}_{b,c,d}=\langle  e_1, e_2, e_3, e_4, 
e_6+b (e_5+ e_8), e_7+c (e_5+ e_8), e_8+d(e_5+ e_8)  \rangle$, $b,c,d \in \mathbb R$,

\medskip
\noindent
5) ${\bf h}_{32}=\langle  e_2+ e_8 \rangle $ and 
${\bf m}_d=\langle  e_1, e_2, e_3, -e_8+2 e_4, e_6, e_7, e_5+d e_8  \rangle$, $d \in \mathbb R$,

\medskip
\noindent
6) ${\bf h}_{35}=\langle  e_6-e_7+ b(e_5+ e_8) \rangle $, $b \ge 0$ and $
\newline
{\bf m}_c=\langle e_1,e_2,e_3,e_4,e_6+e_7,e_5-e_8,e_8-2c e_7+2c b e_8  \rangle$, $c \in \mathbb R$. 
\end{Prop}
\begin{proof} An arbitrary complement ${\bf m}_i$ to the subalgebra ${\bf h}_i$, 
$i=31,\cdots ,35,$ in ${\bf g}$ has as generators in the case $i=31$ 

\medskip
\noindent
\centerline{$e_1+a_1(e_5+a e_8)$, \ $e_2+a_2(e_5+a e_8)$, \ 
$e_3+a_3(e_5+a e_8)$, \  $e_4+a_4(e_5+a e_8)$, }

\medskip
\noindent
\centerline{ $e_6+a_5(e_5+a e_8)$, \  $e_7+a_6(e_5+a e_8)$, \ 
$e_8+a_7(e_5+a e_8)$, }

\medskip
\noindent
in the case $i=32$  

\medskip
\noindent
\centerline{$e_1+a_1(e_2+ e_8)$, \ $e_3+a_2(e_2+ e_8)$, \ $e_4+a_3(e_2+ e_8)$, \ $e_5+a_4(e_2+ e_8)$,}

\medskip
\noindent
\centerline{ $e_6+a_5(e_2+ e_8)$, \ $e_7+a_6(e_2+ e_8)$, \ $e_8+a_7(e_2+ e_8)$, }

\medskip
\noindent
in the case $i=33$  

\medskip
\noindent
\centerline{$e_1+a_1(e_2+ e_3)$, \ $e_3+a_2(e_2+ e_3)$, \ $e_4+a_3(e_2+ e_3)$, \ $e_5+a_4(e_2+ e_3)$,}

\medskip
\noindent
\centerline{ $e_6+a_5(e_2+ e_3)$, \ $e_7+a_6(e_2+ e_3)$, \ $e_8+a_7(e_2+ e_3)$,}

\medskip
\noindent
in the case $i=34$   

\medskip
\noindent
\centerline{$e_1+a_1 e_6$, \  $e_2+a_2 e_6$, \  $e_3+a_3 e_6$, \ $e_4+a_4 e_6$, \  $e_5+a_5 e_6$,}

\medskip
\noindent
\centerline{ $e_7+a_6 e_6$, \ $e_8+a_7 e_6$, }

\medskip
\noindent
in the case $i=35$   

\medskip
\noindent
\centerline{$e_1+a_1(e_6-e_7+b(e_5+ e_8))$, \  $e_2+a_2(e_6-e_7+b(e_5+ e_8))$, }

\medskip
\noindent
\centerline{$e_3+a_3(e_6-e_7+b(e_5+ e_8))$, \ $e_4+a_4(e_6-e_7+b(e_5+ e_8))$, }

\medskip
\noindent
\centerline{$e_5+a_5(e_6-e_7+b(e_5+ e_8))$, \ $e_7+a_6(e_6-e_7+b(e_5+ e_8))$, }

\medskip
\noindent
\centerline{$e_8+a_7(e_6-e_7+b(e_5+ e_8))$, }

\medskip
\noindent
where $a_j \in \mathbb R$, $j=1, \cdots 7$.  Using  the relation 
$[{\bf h_i},{\bf m_i}] \subseteq {\bf m_i}$, $i=31, \cdots ,35,$ we 
obtain the assertion. 
\end{proof}

\bigskip 
\noindent
Now we deal with  the Lie algebra $\mathfrak{su_3}(\mathbb C,1)$. 
It  can be treated as the Lie 
algebra of matrices 
\[(\lambda _1 e_1+\lambda _2 e_2+\lambda _3 e_3+\lambda _4 e_4+\lambda _5 e_5+\lambda _6 e_6+\lambda _7 e_7+\lambda _8 e_8) \mapsto \]
\[ \left ( \begin{array}{ccc}
- \lambda _1 i  & - \lambda _2 -\lambda _3 i & \lambda_4+ \lambda _5 i \\
 \lambda _2 -\lambda _3 i & \lambda _1 i+ \lambda _6 i & \lambda _7 +\lambda _8 i \\
 \lambda _4 -\lambda _5 i & \lambda _7 - \lambda_8 i  & - \lambda _6 i 
\end{array} \right ); \lambda_i \in 
\mathbb R, i=1,\cdots ,8. \] 
Then the  multiplication  of  ${\bf g}$ is given by the 
following: 
\[ [e_1,e_6]=0,\  [e_3,e_2]= 2 e_1,\  [e_4,e_5]= 2(e_1-e_6),\  
[e_8,e_7]=2 e_6, \]
\[ [e_6,e_3]=[e_7,e_4]=[e_8,e_5]=\frac{1}{2} [e_1,e_3]= e_2, \]
\[ [e_2,e_6]=[e_4,e_8]=[e_7,e_5]=\frac{1}{2} [e_2,e_1]=e_3, \]
\[ [e_7,e_2]=[e_3,e_8]=[e_5,e_6]= [e_1,e_5]=e_4, \] 
\[ [e_8,e_2]=[e_7,e_3]=[e_6,e_4]= [e_4,e_1]=e_5, \]
\[ [e_2,e_4]=[e_3,e_5]=[e_8,e_1]=\frac{1}{2} [e_8,e_6]=e_7, \] 
\[ [e_2,e_5]=[e_4,e_3]=[e_1,e_7]=\frac{1}{2} [e_6,e_7]=e_8.  \]
The normalized Cartan-Killing form  
$k:\mathfrak{su_3}(\mathbb C,1) \times 
\mathfrak{su_3}(\mathbb C,1) \to \mathbb R $ is the map 
$(X,Y) \mapsto \frac{1}{12} \hbox{trace} (\hbox{ad}X \hbox{ad}Y)=\frac{1}{2} \hbox{trace} (X Y)$. 
An element 
$X= \lambda _i e_i \in \mathfrak{su_3}(\mathbb C,1)$,  
$\lambda _i \in \mathbb R$, $i=1, \cdots ,8,$  is elliptic, parabolic or 
loxodromic according whether    

\medskip
\noindent
\centerline{$k(X)=k(X,X)=-\lambda _1 ^2-\lambda _2 ^2-\lambda _3 ^2-
\lambda _6 ^2+ \lambda _4 ^2+\lambda _5 ^2+\lambda _7 ^2+\lambda _8 ^2 - 
2 \lambda _1 \lambda _6$  } 
 
\medskip
\noindent
is smaller, equal or greater $0$. 

\medskip
\noindent
Let $H$ be a connected closed subgroup of the group  $PSU_3(\mathbb C,1)$. Then 
according to 
\cite{betten}, Satz 1, p. 251 and \cite{chen}, Section 5, p. 276, the group $H$ 
is,  up to conjugacy, one of the following: 
\newline
(1) $H$ is a subgroup of $Spin _3 \times SO_2(\mathbb R) / \langle (-1,-1) \rangle $, 
\newline
(2) $H$ is a subgroup of the $5$-dimensional solvable group $N G_{1,1}$ in \cite{betten}, 
p. 253, 
\newline
(3) $H$ is the group $SL_2(\mathbb R) \times SO_2(\mathbb R)/ \langle (-1,-1) \rangle $, 
\newline
(4) $H$ is the group $SL_2(\mathbb R) \times \{ 1 \}/ \langle (-1,1) \rangle \cong PSL_2(\mathbb R)$, 
\newline
(5) $H$ is the connected component of the group $SO_3(\mathbb R,1) \cong PSL_2(\mathbb R)$.  
\newline
The Lie algebras ${\bf h}_i$, $i=1, \cdots ,5,$ of  
$H$ in the cases (1) till (5) are given in this order  by 

\medskip
\noindent
\centerline{$\widehat{{\bf h}_1}= \langle e_1,e_2,e_3,e_6 \rangle $,  
$\widehat{{\bf h}_2}=\langle e_1-\frac{1}{2} e_6, e_8, e_4-e_3, e_5+e_2, e_6+e_7 
\rangle $, }
\centerline{$\widehat{{\bf h}_3}= \langle e_1,e_6,e_7,e_8 \rangle $, 
$\widehat{{\bf h}_4}=\langle e_6,e_7,e_8 \rangle$, 
$\widehat{{\bf h}_5}= \langle e_2,e_4,e_7 \rangle $. }

\medskip
\noindent
After a straightforward calculation in $\widehat{{\bf h}_2}$ we obtain that 
the conjugacy 
classes of the 
$4$-dimensional subalgebras of $\mathfrak{su_3}(\mathbb C,1)$ are 
the following:  

\medskip
\noindent  
${\bf h}_1= \langle e_1,e_2,e_3,e_6 \rangle ,$ \ \ 
${\bf h}_2=\langle e_4-e_3, e_2+e_5, e_6+e_7, e_8 \rangle,$ 

\medskip
\noindent
${\bf h}_3=\langle e_1-\frac{1}{2} e_6+a e_8, e_4-e_3, e_2+e_5, e_6+e_7 
\rangle$, \ \ ${\bf h}_4= \langle e_1,e_6,e_7,e_8 \rangle ,$

\medskip
\noindent
where $a \in \mathbb R$.  
\newline
Computations in $\widehat{{\bf h}_1}$ and $\widehat{{\bf h}_2}$ yield that 
the $3$-dimensional subalgebras of $\mathfrak{su_3}(\mathbb C,1)$ have one 
of the following shapes:

\medskip
\noindent  
\centerline{${\bf h}_5=\langle e_1,e_2,e_3 \rangle $, \ \ ${\bf h}_6=\langle e_2,e_4,e_7 \rangle $, \ \ 
${\bf h}_7=\langle e_6,e_7,e_8 \rangle$, } 

\medskip
\noindent  
\centerline{${\bf h}_8=\langle e_5+e_2, e_6+e_7, e_8 \rangle $, 
${\bf h}_9=\langle e_4-e_3+b e_8, e_5+e_2, e_6+e_7 \rangle $, }

\medskip
\noindent  
\centerline{${\bf h}_{10}=\langle e_4-e_3+b (e_5+ e_2), e_6+e_7, e_8+c (e_5+e_2) \rangle $,} 

\medskip
\noindent 
${\bf h}_{11}=\langle e_1-\frac{1}{2} e_6+\frac{3}{2} c(e_4-e_3)-\frac{3}{2} b(e_5+e_2), 
e_8+b (e_4-e_3)+c(e_5+e_2)$,
\newline $e_6+e_7 \rangle $, 
where $b,c \in \mathbb R$. 

\medskip
\noindent 
Similarly we obtain that every $2$-dimensional subalgebra of 
$\mathfrak{su_3}(\mathbb C,1)$ has one 
of the following 
forms: 

\medskip
\noindent  
\centerline{${\bf h}_{12}=\langle e_1,e_6 \rangle $, \ \ 
${\bf h}_{13}=\langle e_4-e_3, e_6+e_7 \rangle$, }

\medskip
\noindent  
\centerline{${\bf h}_{14}=\langle e_5+e_2+b(e_4-e_3),e_6+e_7 \rangle $, \ 
${\bf h}_{15}=\langle e_4-e_3, e_8+b(e_6+e_7) \rangle$, }

\medskip
\noindent
\centerline{${\bf h}_{16}=\langle e_5+e_2+b(e_4-e_3), e_8+c(e_6+e_7) \rangle$, }

\medskip
\noindent
\centerline{${\bf h}_{17}=\langle e_6+e_7,  e_8+b(e_4-e_3)+c(e_5+e_2) \rangle$, }

\medskip
\noindent
\centerline{${\bf h}_{18}=\langle e_6+e_7+a(e_4-e_3)+b(e_5+e_2),  e_8+c(e_4-e_3)+d(e_5+e_2) \rangle$,} 

\medskip
\noindent
\centerline{${\bf h}_{19}=\langle e_1-\frac{1}{2} e_6+a e_8+b(e_4-e_3)+c(e_5+e_2), e_6+e_7 \rangle$, }

\medskip
\noindent
\centerline{${\bf h}_{20}=\langle e_1-\frac{1}{2} e_6+\frac{3}{2} a(e_4-e_3)-\frac{3}{2} b(e_5+e_2)-\frac{3(a^2+b^2)}{2}(e_6+e_7)$, } \centerline{$e_8+b(e_4-e_3)+a(e_5+e_2)+c(e_6+e_7) \rangle $, }

\medskip
\noindent
where $a,b,c,d \in \mathbb R$ and in the Lie algebra ${\bf h}_{18}$ one has 
$bc-ad=\frac{1}{2}$. 

\medskip
\noindent
Moreover, every $1$-dimensional subalgebra  ${\bf h}$ of ${\bf g}$ is given by 

\medskip
\noindent
\centerline{${\bf h}_{21}=\langle e_1+a e_6 \rangle $,  \ \ 
${\bf h}_{22}=\langle e_6 \rangle $, \ \  
${\bf h}_{23}=\langle e_8 \rangle $, } 

\medskip
\noindent
\centerline{${\bf h}_{24}=\langle e_6+e_7+c e_8 \rangle$, \ \ 
${\bf h}_{25}=\langle e_5+e_2+b(e_6+e_7)+c e_8 \rangle$, } 

\medskip
\noindent
\centerline{${\bf h}_{26}=\langle e_4-e_3+a(e_5+e_2)+b(e_6+e_7)+c e_8 \rangle$, }

\medskip
\noindent
\centerline{${\bf h}_{27}=\langle e_1-\frac{1}{2} e_6+d(e_4-e_3)+a(e_5+e_2)+b(e_6+e_7)+c e_8 \rangle$, } 

\medskip
\noindent
where $a,b,c,d$ are real numbers. 

\begin{Prop} \label{su34dim} The Lie algebra $\mathfrak{su_3}(\mathbb C,1)$ 
is reductive 
with a $4$-dimensional 
subalgebra ${\bf h}$ and a complementary subspace ${\bf m}$ generating 
${\bf g}$ if and only if the following holds: 

\medskip
\noindent
1) ${\bf h}_1 \cong \mathfrak{so_3}(\mathbb R) \oplus \mathfrak{so_2}(\mathbb R)= 
\langle e_1,e_2,e_3,e_6 \rangle $ and 
${\bf m}_1=\langle e_4,e_5,e_7,e_8 \rangle $,  

\medskip
\noindent
2) ${\bf h}_4 \cong \mathfrak{sl_2}(\mathbb R) \oplus \mathfrak{so_2}(\mathbb R)= 
\langle e_1,e_6,e_7,e_8 \rangle $ and 
${\bf m}_4=\langle e_2,e_3,e_4,e_5 \rangle $. 
\end{Prop}
\begin{proof}  For the basis elements of an arbitrary complement 
${\bf m}$ to ${\bf h}_1$ in ${\bf g}$ 
we have 

\medskip
\noindent
\centerline{$e_4+a_1 e_1+b_1 e_2+c_1 e_3+d_1 e_6$, \ \ 
$e_5+a_2 e_1+b_2 e_2+c_2 e_3+d_2 e_6$, }

\medskip
\noindent
\centerline{$e_7+a_3 e_1+b_3 e_2+c_3 e_3+d_3 e_6$, \ \ 
$e_8+a_4 e_1+b_4 e_2+c_4 e_3+d_4 e_6$ }

\medskip
\noindent
with the real numbers $a_i, b_i, c_i, d_i$, $i=1,2,3,4$.
\newline
An arbitrary complement ${\bf m}$ to ${\bf h}_2$ in ${\bf g}$ has 
as generators 

\medskip
\noindent
\centerline{$e_1+a_1(e_4-e_3)+b_1(e_5+e_2)+c_1(e_6+e_7)+d_1 e_8$, }

\medskip
\noindent
\centerline{$e_2+a_2(e_4-e_3)+b_2(e_5+e_2)+c_2(e_6+e_7)+d_2 e_8$, } 

\medskip
\noindent
\centerline{$e_3+a_3(e_4-e_3)+b_3(e_5+e_2)+c_3(e_6+e_7)+d_3 e_8$, } 

\medskip
\noindent
\centerline{$e_6+a_4(e_4-e_3)+b_4(e_5+e_2)+c_4(e_6+e_7)+d_4 e_8$, } 

\medskip
\noindent
where $a_i, b_i, c_i, d_i$, $i=1,2,3,4$ are real parameters.
\newline
The basis elements of an arbitrary complement ${\bf m}$ to ${\bf h}_3$ 
in ${\bf g}$ are 

\medskip
\noindent
\centerline{$e_3+a_1(e_1-\frac{1}{2}e_6+ae_8)+b_1(e_4-e_3)+c_1(e_2+e_5)+d_1(e_6+e_7)$, }

\medskip
\noindent
\centerline{$e_5+a_2(e_1-\frac{1}{2}e_6+ae_8)+b_2(e_4-e_3)+c_2(e_2+e_5)+d_2(e_6+e_7)$, }

\medskip
\noindent
\centerline{$e_7+a_3(e_1-\frac{1}{2}e_6+ae_8)+b_3(e_4-e_3)+c_3(e_2+e_5)+d_3(e_6+e_7)$, }

\medskip
\noindent
\centerline{$e_8+a_4(e_1-\frac{1}{2}e_6+ae_8)+b_4(e_4-e_3)+c_4(e_2+e_5)+d_4(e_6+e_7)$, }

\medskip
\noindent
where $a_i, b_i, c_i, d_i \in \mathbb R$, $i=1,2,3,4$. 
\newline
As the generators of an arbitrary complement ${\bf m}$ to ${\bf h}_4$ 
in ${\bf g}$ we can choose the following: 

\medskip
\noindent
\centerline{$e_2+a_1 e_1+b_1 e_6+c_1 e_7+d_1 e_8$, \ \ 
$e_3+a_2 e_1+b_2 e_6+c_2 e_7+d_2 e_8$, }

\medskip
\noindent
\centerline{$e_4+a_3 e_1+b_3 e_6+c_3 e_7+d_3 e_8$, \ \ 
$e_5+a_4 e_1+b_4 e_6+c_4 e_6+d_4 e_8$, }

\medskip
\noindent
where $a_i, b_i, c_i, d_i$, $i=1,2,3,4$ are  real numbers.
\newline
Now the assertion follows from the relation 
$[{\bf h},{\bf m}] \subseteq {\bf m}$.  
\end{proof}

\begin{Prop} \label{su33dim} The Lie algebra ${\bf g}=\mathfrak{su_3}(\mathbb C,1)$ is 
reductive with respect to precisely one of the 
following pairs $({\bf h}, {\bf m})$, where ${\bf h}$ is 
a $3$-dimensional 
subalgebra of ${\bf g}$ and ${\bf m}$ is a  complementary  subspace  to ${\bf h}$ 
generating  ${\bf g}$: 

\medskip
\noindent 
1) ${\bf h}_6 \cong \mathfrak{sl_2}(\mathbb R)=\langle e_2, e_4, e_7 \rangle $ and 
${\bf m}_6=\langle e_1, e_3, e_5, e_6, e_8 \rangle$, 

\medskip
\noindent 
2) ${\bf h}_7 \cong \mathfrak{sl_2}(\mathbb R)=\langle e_6,e_7,e_8 \rangle $ and  
${\bf m}_7=\langle e_1-\frac{1}{2} e_6, e_2, e_3, e_4, e_5 \rangle$. 
\end{Prop}
\begin{proof} An arbitrary complement ${\bf m}_i$ to the subalgebra ${\bf h}_i$, 
$i=5,\cdots ,11,$ in 
${\bf g}$ has as generators 
in the case $i=5$ 

\medskip
\noindent
\centerline{$e_4+a_1 e_1 +b_1 e_2+c_1 e_3$, \ \ $e_5+a_2 e_1 +b_2 e_2+c_2 e_3$, \ \ $e_6+a_3 e_1 +b_3 e_2+c_3 e_3$,}    

\medskip
\noindent
\centerline{$e_7+a_4 e_1 +b_4 e_2+c_4 e_3$, \ \ $e_8+a_5 e_1 +b_5 e_2+c_5 e_3$, }

\medskip
\noindent
in the case $i=6$

\medskip
\noindent
\centerline{$e_1+a_1 e_2 +b_1 e_4+c_1 e_7$, \ \ $e_3+a_2 e_2 +b_2 e_4+c_2 e_7$, \ \ $e_5+a_3 e_2 +b_3 e_4+c_3 e_7$,}    

\medskip
\noindent
\centerline{$e_6+a_4 e_2 +b_4 e_4+c_4 e_7$, \ \ $e_8+a_5 e_2 +b_5 e_4+c_5 e_7$, }

\medskip
\noindent
in the case $i=7$

\medskip
\noindent
\centerline{$e_1+a_1 e_6 +b_1 e_7+c_1 e_8$, \ \ $e_2+a_2 e_6 +b_2 e_7+c_2 e_8$, \ \ $e_3+a_3 e_6 +b_3 e_7+c_3 e_8$,}    

\medskip
\noindent
\centerline{$e_4+a_4 e_6 +b_4 e_7+c_4 e_8$, \ \ $e_5+a_5 e_6 +b_5 e_7+c_5 e_8$, }

\medskip
\noindent
in the case $i=8$

\medskip
\noindent
\centerline{$e_1+a_1 (e_2+e_5) +b_1 (e_6+e_7)+c_1 e_8$, \ \ $e_2+a_2 (e_2+e_5) +b_2 (e_6+e_7)+c_2 e_8$, }    

\medskip
\noindent
\centerline{$e_3+a_3 (e_2+e_5) +b_3 (e_6+e_7)+c_3 e_8$, \ \ $e_4+a_4 (e_2+e_5) +b_4 (e_6+e_7)+c_4 e_8$, } 

\medskip
\noindent
\centerline{$e_6+a_5 (e_2+e_5) +b_5 (e_6+e_7)+c_5 e_8$, }

\medskip
\noindent
in the case $i=9$

\medskip
\noindent
\centerline{$e_1+a_1 (e_2+e_5) +b_1 (e_6+e_7)+c_1 (e_4-e_3+b e_8)$, }

\medskip
\noindent
\centerline{$e_2+a_2 (e_2+e_5) +b_2 (e_6+e_7)+c_2 (e_4-e_3+b e_8)$, }

\medskip
\noindent
\centerline{$e_3+a_3 (e_2+e_5) +b_3 (e_6+e_7)+c_3 (e_4-e_3+b e_8)$, } 

\medskip
\noindent
\centerline{$e_6+a_4 (e_2+e_5) +b_4 (e_6+e_7)+c_4 (e_4-e_3+b e_8)$, }

\medskip
\noindent
\centerline{$e_8+a_5 (e_2+e_5) +b_5 (e_6+e_7)+c_5 (e_4-e_3+b e_8)$, }

\medskip
\noindent
in the case $i=10$ 

\medskip
\noindent
\centerline{$e_1+a_1 (e_4-e_3+b(e_2+e_5)) +b_1 (e_6+e_7)+c_1 (e_8 +c(e_2+e_5))$, }

\medskip
\noindent
\centerline{$e_2+a_2 (e_4-e_3+b(e_2+e_5)) +b_2 (e_6+e_7)+c_2 (e_8 +c(e_2+e_5))$, }

\medskip
\noindent
\centerline{$e_3+a_3 (e_4-e_3+b(e_2+e_5)) +b_3 (e_6+e_7)+c_3 (e_8 +c(e_2+e_5))$, }

\medskip
\noindent
\centerline{$e_5+a_4 (e_4-e_3+b(e_2+e_5)) +b_4 (e_6+e_7)+c_4 (e_8 +c(e_2+e_5))$, }

\medskip
\noindent
\centerline{$e_6+a_5 (e_4-e_3+b(e_2+e_5)) +b_5 (e_6+e_7)+c_5 (e_8 +c(e_2+e_5))$,}

\medskip
\noindent
and in the case $i=11$ 

\medskip
\noindent
\centerline{$e_2+a_1(e_1-\frac{1}{2} e_6+\frac{3}{2} c(e_4-e_3)-\frac{3}{2} b(e_5+e_2))+ $}
\newline
\centerline{$b_1(e_8+b (e_4-e_3)+c(e_5+e_2))+c_1(e_6+e_7)$, }

\medskip
\noindent
\centerline{$e_3+a_2(e_1-\frac{1}{2} e_6+\frac{3}{2} c(e_4-e_3)-\frac{3}{2} b(e_5+e_2))+ $}
\newline
\centerline{$b_2(e_8+b (e_4-e_3)+c(e_5+e_2))+c_2(e_6+e_7)$, }

\medskip
\noindent
\centerline{$e_4+a_3(e_1-\frac{1}{2} e_6+\frac{3}{2} c(e_4-e_3)-\frac{3}{2} b(e_5+e_2))+ $}
\newline
\centerline{$b_3(e_8+b (e_4-e_3)+c(e_5+e_2))+c_3(e_6+e_7)$, }

\medskip
\noindent
\centerline{$e_5+a_4(e_1-\frac{1}{2} e_6+\frac{3}{2} c(e_4-e_3)-\frac{3}{2} b(e_5+e_2))+ $}
\newline
\centerline{$b_4(e_8+b (e_4-e_3)+c(e_5+e_2))+c_4(e_6+e_7)$, }

\medskip
\noindent
\centerline{$e_7+a_5(e_1-\frac{1}{2} e_6+\frac{3}{2} c(e_4-e_3)-\frac{3}{2} b(e_5+e_2))+ $}
\newline
\centerline{$b_5(e_8+b (e_4-e_3)+c(e_5+e_2))+c_5(e_6+e_7)$, }

\medskip
\noindent
where $a_j, b_j, c_j \in \mathbb R$, $j=1, \cdots , 5$.  
The relation $[{\bf h_i},{\bf m_i}] \subseteq {\bf m_i}$, $i=5, \cdots ,11,$ 
yields the assertion. 
\end{proof}

\begin{Prop} \label{su32dim} The Lie algebra ${\bf g}=\mathfrak{su_3}(\mathbb C,1)$ is 
reductive with respect to 
the following pairs $({\bf h}, {\bf m})$, where ${\bf h}$ is 
a $2$-dimensional subalgebra of ${\bf g}$ and  ${\bf m}$ is a  complementary  
subspace to ${\bf h}$ generating  ${\bf g}$, if and only if one of the following holds: 

\medskip
\noindent 
1) ${\bf h}_{12}=\langle e_1, e_6 \rangle$  and 
${\bf m}_{12}=\langle e_2, e_3, e_4, e_5, e_7, e_8 \rangle$, 

\medskip
\noindent 
2) ${\bf h}_{20}=\langle e_1-\frac{1}{2} e_6+\frac{3}{2} a(e_4-e_3)-\frac{3}{2} b(e_5+e_2)-\frac{3(a^2+b^2)}{2}(e_6+e_7)$, 
\newline
\centerline{$e_8+b(e_4-e_3)+a(e_5+e_2)+c(e_6+e_7) \rangle $  
 } and 
\newline
${\bf m}_{20}=\langle e_6+e_7, e_2+e_5, e_4-e_3, e_4- b e_8+2 a e_1- a e_6, 
e_2+ a e_8+2 b e_1- b e_6,$ 
\newline
$e_6+ c e_8+ b e_5- a e_4 \rangle $, $a,b,c \in \mathbb R$. 
\end{Prop} 
\begin{proof} An arbitrary complement ${\bf m}_i$ to the subalgebra ${\bf h}_i$, 
$i=12,\cdots ,20,$ in 
${\bf g}$ has as generators in the case $i=12$ 

\medskip
\noindent
\centerline{$e_2+a_1 e_1+b_1 e_6$, \ \ $e_3+a_2 e_1+b_2 e_6$, \ \ $e_4+a_3 e_1+b_3 e_6$, }

\medskip
\noindent
\centerline{$e_5+a_4 e_1+b_4 e_6$, \ \ $e_7+a_5 e_1+b_5 e_6$, \ \ $e_8+a_6 e_1+b_6 e_6$, }

\medskip
\noindent
in the case $i=13$ 

\medskip
\noindent
\centerline{$e_1+a_1 (e_4-e_3)+b_1 (e_6+e_7)$, \ \ $e_2+a_2 (e_4-e_3)+b_2 (e_6+e_7)$,}

\medskip
\noindent
\centerline{$e_3+a_3 (e_4-e_3)+b_3 (e_6+e_7)$, \ \ $e_5+a_4 (e_4-e_3)+b_4 (e_6+e_7)$,}

\medskip
\noindent
\centerline{$e_6+a_5 (e_4-e_3)+b_5 (e_6+e_7)$, \ \ $e_8+a_6 (e_4-e_3)+b_6 (e_6+e_7)$,}

\medskip
\noindent
in the case $i=14$ 

\medskip
\noindent
\centerline{$e_1+a_1 (e_2+e_5+b(e_4-e_3))+b_1 (e_6+e_7)$, } 

\medskip
\noindent
\centerline{$e_2+a_2 (e_2+e_5+b(e_4-e_3))+b_2 (e_6+e_7)$,}

\medskip
\noindent
\centerline{$e_3+a_3 (e_2+e_5+b(e_4-e_3))+b_3 (e_6+e_7)$, }

\medskip
\noindent
\centerline{$e_4+a_4 (e_2+e_5+b(e_4-e_3))+b_4 (e_6+e_7)$,}

\medskip
\noindent
\centerline{$e_6+a_5 (e_2+e_5+b(e_4-e_3))+b_5 (e_6+e_7)$, }

\medskip
\noindent
\centerline{$e_8+a_6 (e_2+e_5+b(e_4-e_3))+b_6 (e_6+e_7)$, }

\medskip
\noindent
in the case $i=15$ 

\medskip
\noindent
\centerline{$e_1+a_1 (e_4-e_3)+b_1 (e_8+b(e_6+e_7))$, }

\medskip
\noindent
\centerline{$e_2+a_2 (e_4-e_3)+b_2 (e_8+b(e_6+e_7))$, }

\medskip
\noindent
\centerline{$e_3+a_3 (e_4-e_3)+b_3 (e_8+b(e_6+e_7))$, }

\medskip
\noindent
\centerline{$e_5+a_4 (e_4-e_3)+b_4 (e_8+b(e_6+e_7))$,}

\medskip
\noindent
\centerline{$e_6+a_5 (e_4-e_3)+b_5 (e_8+b(e_6+e_7))$, }

\medskip
\noindent
\centerline{$e_7+a_6 (e_4-e_3)+b_6 (e_8+b(e_6+e_7))$,}

\medskip
\noindent
in the case $i=16$ 

\medskip
\noindent
\centerline{$e_1+a_1 (e_5+e_2+b(e_4-e_3))+b_1 (e_8+c(e_6+e_7))$, }

\medskip
\noindent
\centerline{$e_2+a_2 (e_5+e_2+b(e_4-e_3))+b_2 (e_8+c(e_6+e_7))$,}

\medskip
\noindent
\centerline{$e_3+a_3 (e_5+e_2+b(e_4-e_3))+b_3 (e_8+c(e_6+e_7))$, }

\medskip
\noindent
\centerline{$e_4+a_4 (e_5+e_2+b(e_4-e_3))+b_4 (e_8+c(e_6+e_7))$,}

\medskip
\noindent
\centerline{$e_6+a_5 (e_5+e_2+b(e_4-e_3))+b_5 (e_8+b(e_6+e_7))$, }

\medskip
\noindent
\centerline{$e_7+a_6 (e_5+e_2+b(e_4-e_3))+b_6 (e_8+b(e_6+e_7))$,}

\medskip
\noindent
in the case $i=17$  

\medskip
\noindent
\centerline{$e_1+a_1 (e_6+e_7)+b_1 (e_8+b(e_4-e_3)+c(e_5+e_2))$, }

\medskip
\noindent
\centerline{$e_2+a_2 (e_6+e_7)+b_2 (e_8+b(e_4-e_3)+c(e_5+e_2))$, }

\medskip
\noindent
\centerline{$e_3+a_3 (e_6+e_7)+b_3 (e_8+b(e_4-e_3)+c(e_5+e_2))$, }

\medskip
\noindent
\centerline{$e_4+a_4 (e_6+e_7)+b_4 (e_8+b(e_4-e_3)+c(e_5+e_2))$,}

\medskip
\noindent
\centerline{$e_5+a_5 (e_6+e_7)+b_5 (e_8+b(e_4-e_3)+c(e_5+e_2))$, }

\medskip
\noindent
\centerline{$e_6+a_6 (e_6+e_7)+b_6 (e_8+b(e_4-e_3)+c(e_5+e_2))$, }

\medskip
\noindent
in the case $i=18$  

\medskip
\noindent
$e_1+a_1 (e_6+e_7+a(e_4-e_3)+b(e_5+e_2))+b_1 (e_8+c(e_4-e_3)+d(e_5+e_2))$, 

\medskip
\noindent
$e_2+a_2 (e_6+e_7+a(e_4-e_3)+b(e_5+e_2))+b_2 (e_8+c(e_4-e_3)+d(e_5+e_2))$, 

\medskip
\noindent
$e_3+a_3 (e_6+e_7+a(e_4-e_3)+b(e_5+e_2))+b_3 (e_8+c(e_4-e_3)+d(e_5+e_2))$, 

\medskip
\noindent
$e_4+a_4 (e_6+e_7+a(e_4-e_3)+b(e_5+e_2))+b_4 (e_8+c(e_4-e_3)+d(e_5+e_2))$, 

\medskip
\noindent
$e_5+a_5 (e_6+e_7+a(e_4-e_3)+b(e_5+e_2))+b_5 (e_8+c(e_4-e_3)+d(e_5+e_2))$, 

\medskip
\noindent
$e_6+a_6 (e_6+e_7+a(e_4-e_3)+b(e_5+e_2))+b_6 (e_8+c(e_4-e_3)+d(e_5+e_2))$, 

\medskip
\noindent
in the case $i=19$ 

\medskip
\noindent
\centerline{$e_2+a_1 (e_1- \frac{1}{2}e_6+a e_8+b(e_4-e_3)+c(e_5+e_2))+b_1 (e_6+e_7) $, }  

\medskip
\noindent
\centerline{$e_3+a_2 (e_1- \frac{1}{2}e_6+a e_8+b(e_4-e_3)+c(e_5+e_2))+b_2 (e_6+e_7) $, }  

\medskip
\noindent
\centerline{$e_4+a_3 (e_1- \frac{1}{2}e_6+a e_8+b(e_4-e_3)+c(e_5+e_2))+b_3 (e_6+e_7) $, }  

\medskip
\noindent
\centerline{$e_5+a_4 (e_1- \frac{1}{2}e_6+a e_8+b(e_4-e_3)+c(e_5+e_2))+b_4 (e_6+e_7) $, } 

\medskip
\noindent
\centerline{$e_7+a_5 (e_1- \frac{1}{2}e_6+a e_8+b(e_4-e_3)+c(e_5+e_2))+b_5 (e_6+e_7) $, } 

\medskip
\noindent
\centerline{$e_8+a_6 (e_1- \frac{1}{2}e_6+a e_8+b(e_4-e_3)+c(e_5+e_2))+b_6 (e_6+e_7) $, } 

\medskip
\noindent
and in the case $i=20$

\medskip
\noindent
$e_2+a_1 (e_1- \frac{1}{2}e_6+ \frac{3}{2}a(e_4-e_3)-\frac {3}{2}b(e_5+e_2)-\frac{3(a^2+b^2)}{2} (e_6+e_7))+ $  
\newline
$+b_1(e_8+b(e_4-e_3)+a(e_5+e_2)+c(e_6+e_7))$, 

\medskip
\noindent
$e_3+a_2 (e_1- \frac{1}{2}e_6+ \frac{3}{2}a(e_4-e_3)-\frac {3}{2}b(e_5+e_2)-\frac{3(a^2+b^2)}{2} (e_6+e_7))+ $  
\newline
$+b_2(e_8+b(e_4-e_3)+a(e_5+e_2)+c(e_6+e_7))$,

\medskip
\noindent
$e_4+a_3 (e_1- \frac{1}{2}e_6+ \frac{3}{2}a(e_4-e_3)-\frac {3}{2}b(e_5+e_2)-\frac{3(a^2+b^2)}{2} (e_6+e_7))+ $ 
\newline
$+b_3(e_8+b(e_4-e_3)+a(e_5+e_2)+c(e_6+e_7))$, 

\medskip
\noindent
$e_5+a_4 (e_1- \frac{1}{2}e_6+ \frac{3}{2}a(e_4-e_3)-\frac {3}{2}b(e_5+e_2)-\frac{3(a^2+b^2)}{2} (e_6+e_7))+ $ 
\newline
$+b_4(e_8+b(e_4-e_3)+a(e_5+e_2)+c(e_6+e_7))$, 

\medskip
\noindent
$e_6+a_5 (e_1- \frac{1}{2}e_6+ \frac{3}{2}a(e_4-e_3)-\frac {3}{2}b(e_5+e_2)-\frac{3(a^2+b^2)}{2} (e_6+e_7))+ $ 
\newline
$+b_5(e_8+b(e_4-e_3)+a(e_5+e_2)+c(e_6+e_7))$, 

\medskip
\noindent
$e_7+a_6 (e_1- \frac{1}{2}e_6+ \frac{3}{2}a(e_4-e_3)-\frac {3}{2}b(e_5+e_2)-\frac{3(a^2+b^2)}{2} (e_6+e_7))+ $ 
\newline
$+b_6(e_8+b(e_4-e_3)+a(e_5+e_2)+c(e_6+e_7))$, 

\medskip
\noindent
where $a_j, b_j \in \mathbb R$, $j=1, \cdots 6$.  
Using the  relation $[{\bf h_i},{\bf m_i}] \subseteq {\bf m_i}$, 
$i=12, \cdots ,20,$ 
we obtain the assertion. 
\end{proof}

\begin{Prop} \label{su31dim} The Lie algebra ${\bf g}=\mathfrak{su_3}(\mathbb C,1)$ 
is reductive 
with a $1$-dimensional subalgebra ${\bf h}$ and a $7$-dimensional 
complementary subspace ${\bf m}$ generating ${\bf g}$ in precisely one of 
the following cases:

\medskip
\noindent 
1) ${\bf h}=\langle e_1-2 e_6 \rangle $, 
${\bf m}_{b,c,d}=\langle e_2+b(e_1-2 e_6), e_3+c(e_1-2 e_6), 
e_6+d(e_1-2 e_6), $ 
\newline
$e_4, e_5, e_7, e_8 \rangle $, 
when $b,c,d \in \mathbb R$,  

\medskip
\noindent 
2) ${\bf h}=\langle e_1+ e_6 \rangle $ and 
${\bf m}_{b,c,d}=\langle e_2, e_3, e_7, e_8, e_4+d(e_1+ e_6), 
e_5+b(e_1+ e_6), $
\newline
$e_6+c(e_1+ e_6) \rangle $ with  
$b,c,d \in \mathbb R$,  

\medskip
\noindent 
3) ${\bf h}=\langle e_1-\frac{1}{2} e_6 \rangle $, 
${\bf m}_{b,c,d}=\langle e_2, e_3, e_4, e_5, e_6+b(e_1-\frac{1}{2} e_6), 
e_7+c(e_1-\frac{1}{2} e_6),$
\newline
$e_8+d(e_1-\frac{1}{2} e_6) \rangle $ and  $b,c,d \in \mathbb R$,

\medskip
\noindent
4) ${\bf h}_a=\langle e_1+a e_6 \rangle $ and ${\bf m}_{b}=\langle e_2, e_3, e_4, e_5, e_6+b(e_1+a e_6), e_7, e_8 \rangle $, 
\newline
where $a \in \mathbb R \backslash \{ -\frac{1}{2}, -2,1 \}$, $b,c,d \in \mathbb R$,

\medskip
\noindent 
5) ${\bf h}=\langle e_6 \rangle $ and ${\bf m}_a=\langle e_1+a e_6, e_2, e_3, e_4, e_5, e_7, e_8 \rangle $, $a \in \mathbb R$,

\medskip
\noindent 
6) ${\bf h}=\langle e_8 \rangle $ and ${\bf m}_a=\langle e_1+a e_8, e_2, e_3, e_4, e_5, e_6, e_7 \rangle $, $a \in \mathbb R$,

\medskip
\noindent 
7) ${\bf h}=\langle e_6+e_7+c e_8 \rangle $  and 
${\bf m}_b=\langle e_1+b c e_8, e_2, e_3, e_4, e_5, e_6+e_7, e_7-\frac{1}{c} e_8  \rangle $ with  $c \in \mathbb R \backslash \{ 0 \}$, $b \in \mathbb R$,   

\medskip
\noindent 
8) ${\bf h}_{b,c}=\langle e_5+e_2+b(e_6+e_7)+c e_8 \rangle $ and 
${\bf m}_d=\langle e_1-\frac{c^3 d-c d-b}{2 c} e_8, e_2+\frac{1}{c} e_8, $
\newline
$e_3+c d e_8,  e_7-\frac{b+c d}{c} e_8, e_4-e_3, e_2+ e_5, e_6+e_7  
\rangle $, where  $b,c,d \in \mathbb R, c \neq 0$,    

\medskip
\noindent
9) ${\bf h}_{a,b,c}=\langle e_4-e_3+a(e_5+e_2)+b(e_6+e_7)+c e_8 \rangle $  and 
\newline
${\bf m}_d=\langle e_2-d c e_8, e_3-\frac{1+d c^2 a+a^2}{c} e_8,  
e_6-\frac{a^3+a-b c +d c^2+d c^2 a^2}{c^2} e_8, e_5+e_2, e_6+e_7, $
\newline
$e_4-e_3, 
e_1+\frac{b c+c^2 a-a-a^3+c^4 d-c^2 d-c^2 a^2 d}{2 c^2} e_8 \rangle $   
with $a,b,c,d \in \mathbb R, c \neq 0$,  

\medskip
\noindent
10) ${\bf h}_{a,b,c,d}=\langle e_1-\frac{1}{2} e_6+d(e_4-e_3)+a(e_5+e_2)+b(e_6+e_7)+c e_8 \rangle$ and ${\bf m}_f=\langle e_6+e_7, e_4-e_3, e_5+e_2, e_3+f(e_1-\frac{1}{2} e_6+c e_8), e_2-\frac{2c}{3} e_4-\frac{4a}{3} e_1-\frac{2a}{3} e_7+\frac{2d}{3} e_8, 
e_7-\frac{b}{c} e_8+\frac{a}{c} e_4+\frac{d}{c} e_2, e_8-\frac{8 a c-4 f c^2+24 f d^2-9 f+12 d}{2(8 d c-3 a+4 a c^2)} (e_1-\frac{1}{2} e_6+c e_8)  \rangle$,   

\medskip
\noindent
where $a,b,c,d,f \in \mathbb R$, $c \neq 0$, $8 d c-3 a+4 a c^2 \neq 0$. 
\end{Prop}
\begin{proof} An arbitrary complement ${\bf m}_i$ to the subalgebra ${\bf h}_i$, 
$i=21, \cdots ,27,$ in ${\bf g}$ has as generators in the case $i=21$ 

\medskip
\noindent
\centerline{$e_2+a_1(e_1+a e_6)$,\ \ $e_3+a_2(e_1+a e_6)$,\ \ 
$e_4+a_3 (e_1+a e_6)$, $e_5+a_4 (e_1+a e_6)$, } 

\medskip
\noindent
\centerline{$e_6+a_5 (e_1+a e_6)$,\ \ $e_7+a_6 (e_1+a e_6)$, \ \ 
$e_8+a_7 (e_1+a e_6)$, } 

\medskip
\noindent
in the case $i=22$

\medskip
\noindent
\centerline{$e_1+a_1 e_6$,\ \ $e_2+a_2 e_6$,\ \ $e_3+a_3 e_6$, $e_4+a_4 e_6$, } 

\medskip
\noindent
\centerline{$e_5+a_5 e_6$,\ \ $e_7+a_6 e_6$, \ \ $e_8+a_7 e_6$, } 

\medskip
\noindent
in the case $i=23$

\medskip
\noindent
\centerline{$e_1+a_1 e_8$,\ \ $e_2+a_2 e_8$,\ \ $e_3+a_3 e_8$, $e_4+a_4 e_8$, } 

\medskip
\noindent
\centerline{$e_5+a_5 e_8$,\ \ $e_6+a_6 e_8$, \ \ $e_7+a_7 e_8$, } 

\medskip
\noindent
in the case $i=24$

\medskip
\noindent
\centerline{$e_1+a_1(e_6+ e_7+c e_8)$,\ \ 
$e_2+a_2(e_6+e_7+c e_8)$, }

\medskip
\noindent
\centerline{$e_3+a_3(e_6+e_7+c e_8)$, $e_4+a_4(e_6+e_7+c e_8)$, } 

\medskip
\noindent
\centerline{$e_5+a_5(e_6+e_7+c e_8)$,\ \ $e_7+a_6(e_6+e_7+c e_8)$,}

\medskip
\noindent
\centerline{$e_8+a_7(e_6+e_7+c e_8)$, } 

\medskip
\noindent
in the case $i=25$ 

\medskip
\noindent
\centerline{$e_1+a_1(e_5+e_2+b(e_6+ e_7)+c e_8)$, \ $e_2+a_2(e_5+e_2+b(e_6+e_7)+c e_8)$,}

\medskip
\noindent
\centerline{$e_3+a_3(e_5+e_2+b(e_6+e_7)+c e_8)$, \ $e_4+a_4(e_5+e_2+b(e_6+e_7)+c e_8)$, } 

\medskip
\noindent
\centerline{$e_6+a_5(e_5+e_2+b(e_6+e_7)+c e_8)$, \ $e_7+a_6(e_5+e_2+b(e_6+e_7)+c e_8)$, }

\medskip
\noindent
\centerline{$e_8+a_7(e_5+e_2+b(e_6+e_7)+c e_8)$, } 

\medskip
\noindent
in the case $i=26$ 

\medskip
\noindent
\centerline{$e_1+a_1(e_4-e_3+a(e_5+e_2)+b(e_6+ e_7)+c e_8)$,}

\medskip
\noindent
\centerline{$e_2+a_2(e_4-e_3+a(e_5+e_2)+b(e_6+e_7)+c e_8)$,}

\medskip
\noindent
\centerline{$e_3+a_3(e_4-e_3+a(e_5+e_2)+b(e_6+e_7)+c e_8)$,}

\medskip
\noindent
\centerline{$e_5+a_4(e_4-e_3+a(e_5+e_2)+b(e_6+e_7)+c e_8)$, } 

\medskip
\noindent
\centerline{$e_6+a_5(e_4-e_3+a(e_5+e_2)+b(e_6+e_7)+c e_8)$, }

\medskip
\noindent
\centerline{$e_7+a_6(e_4-e_3+a(e_5+e_2)+b(e_6+e_7)+c e_8)$, }

\medskip
\noindent
\centerline{$e_8+a_6(e_4-e_3+a(e_5+e_2)+b(e_6+e_7)+c e_8)$, }

\medskip
\noindent
in the case $i=27$ 

\medskip
\noindent
\centerline{$e_2+a_1(e_1- \frac{1}{2} e_6+ d(e_4-e_3)+a(e_5+e_2)+b(e_6+e_7)+
c e_8)$,}

\medskip
\noindent
\centerline{$e_3+a_2(e_1- \frac{1}{2} e_6+d(e_4-e_3)+a(e_5+e_2)+b(e_6+e_7)+
c e_8)$,}

\medskip
\noindent
\centerline{$e_4+a_3(e_1- \frac{1}{2} e_6+d(e_4-e_3)+a(e_5+e_2)+b(e_6+e_7)+
c e_8)$,}  

\medskip
\noindent
\centerline{$e_5+a_4(e_1- \frac{1}{2} e_6+ d(e_4-e_3)+a(e_5+e_2)+b(e_6+e_7)+
c e_8)$, }          

\medskip
\noindent
\centerline{$e_6+a_5(e_1- \frac{1}{2} e_6+ d(e_4-e_3)+a(e_5+e_2)+b(e_6+e_7)+
c e_8)$, } 

\medskip
\noindent
\centerline{$e_7+a_6(e_1- \frac{1}{2} e_6+d(e_4-e_3)+a(e_5+e_2)+b(e_6+e_7)+
c e_8)$, } 

\medskip
\noindent
\centerline{$e_8+a_7(e_1- \frac{1}{2} e_6+ d(e_4-e_3)+a(e_5+e_2)+b(e_6+e_7)+
c e_8)$, } 

\medskip
\noindent
where $a_j$, $j=1, \cdots ,7,$ are real parameters. 
The relation $[{\bf h_i},{\bf m_i}] \subseteq {\bf m_i}$, $i=21, \cdots ,27,$ 
yields  the assertion. 
\end{proof}

\section{ Left A-loops as sections in simple Lie groups}

\noindent
The connected almost differentiable left A-loops $L$ with $\hbox{dim} \ L \le 2$ are classified in 
\cite{loops}, Section 27 and Theorem 18.14. Furthermore, all $3$-dimensional left A-loops which are 
differentiable sections in a non-solvable Lie group are determined in \cite{figula3}. 
In this section we deal with the  at least $4$-dimensional almost 
differentiable left A-loops having an at most $9$-dimensional simple Lie 
group $G$ as the group topologically generated by their left translations. 
According to Lemma \ref{kompakt} the group $G$ is not compact.   

\begin{Prop} There exists no at least  $4$-dimensional differentiable 
left A-loop 
having a group locally isomorphic to $PSL_2(\mathbb C)$ as the group 
topologically generated by its left translations.
\end{Prop}
\begin{proof}
Since the tangent space $T_e L$ for an almost differentiable left A-loop $L$ is reductive 
only the pairs $({\bf h}, {\bf m})$ in Proposition \ref{sl2c} can  occur 
as the tangent objects 
$(T_1 H, T_e L)$, where $H$ is the stabilizer of the identity $e$ of $L$. 
A  maximal compact subalgebra of the Lie algebra 
${\bf h}_3$ as well as of ${\bf h}_6$ is isomorphic to $so_2(\mathbb R)$. Hence  
the Lie group corresponding to  ${\bf h}_3$ as well as to 
 ${\bf h}_6$ cannot be the 
stabilizer of $e \in L$ (cf. Lemma \ref{fundamenta}).  
Moreover, the hyperbolic elements $e_1 \in {\bf h}_4$ and 
$e_2 \in {\bf m}_a$ are conjugate (see {\bf 1.1}).  This  
 contradiction to Lemma \ref{conjugate} yields the assertion.  
\end{proof}

\begin{Prop} Let  $G$ be locally isomorphic to 
$SL_3(\mathbb R)$. 
Every  connected almost differentiable 
left A-loop having $G$  as the group topologically 
generated by its left translations 
is isomorphic to the $5$-dimensional Bruck loop $L_0$ of hyperbolic type 
having the group $SO_3(\mathbb R)$ as the stabilizer of $e \in L_0$.  
\end{Prop}
\begin{proof} 
Since the tangent space $T_e L$ for an almost  differentiable left A-loop 
$L$ 
is reductive we have to investigate the pairs $({\bf h}, {\bf m})$ 
listed in Propositions 
\ref{sl3R4dim}, \ref{sl3R3dim}, \ref{sl3R2dim} and \ref{sl3R1dim}. 
According to Lemma \ref{fundamenta} the Lie groups belonging to 
 the Lie algebras ${\bf h}_5$, 
 ${\bf h}_7$,  ${\bf h}_8$,  ${\bf h}_{30}$ and  ${\bf h}_{35}$ for $b=0$ 
cannot be  stabilizers of $e \in L$. 
The element $-e_5+e_8 \in {\bf h}_{26}$ 
is conjugate to $\frac{1}{2} e_1+2 e_3 \in {\bf m}_{26}$ 
under  $g=\left (\begin{array}{rrr}
0 & 0 & 1 \\
1 & -\frac{1}{2} & 0 \\
1 & \frac{1}{2} & 0 \end{array} \right)$, the element 
$e_2+e_8 \in {\bf h}_{32}$ is 
conjugate to 
$e_1+2 e_7-e_8+2 e_4 \in {\bf m}_d$ under 
$g=\left ( \begin{array}{rrr}
0 & 0 & 1 \\
0 & -\frac{1}{2} & 0 \\
2 & 2 & 0 \end{array} \right )$ and $e_6-e_7+b(e_5+e_8) \in {\bf h}_{35}$, 
$b>0$, is conjugate to 
$(b^2+1)e_1-e_3+2b(e_5-e_8) \in {\bf m}_c$ under 
$g=\left ( \begin{array}{rrr}
0 & 0 & 1 \\
1 & -b & 0 \\
0 & 1 & 0 \end{array} \right )$. Moreover, the element 
$e_8+\frac{1}{a} e_5 \in {\bf h}_{31,1}$ 
is conjugate to 
$\frac{-a^2+a+1}{a^2} e_1+e_2+e_3+e_4-e_6-e_7 \in {\bf m}_b$ 
 under   
$g=\left ( \begin{array}{ccc}
1 & -\frac{1}{a} & -1 \\
1 & \frac{a+1}{a} & -1 \\
0 & \frac{a}{2+a} & \frac{a}{2+a} \end{array} \right )$. 

\smallskip
\noindent
In the case 2) of Proposition \ref{sl3R1dim} we choose 
$k \in \mathbb R \backslash \{ 0 \}$ in such a way 
that $l:=k^2 c+k+b \neq 0$.  
Then the element $l (e_5-2 e_8) \in {\bf h}_{31,2}$ is conjugate to 
\newline
$e_1+b(e_5-2 e_8)+3 l(e_2-k e_6)+k(e_3+c(e_5-2e_8))+\frac{3k^2c+k+3b}{3 l} (k e_4-e_7) \in $
${\bf m}_{b,c,d}$ 
 under $g=\left ( \begin{array}{ccc}
0 & -\frac{3k^2 c+k+3b}{3k l} & 1 \\
k & 1 & 0 \\
\frac{-k}{3 l} & \frac{1}{3 l} & 1 \end{array} \right )$. 

\smallskip
\noindent 
In the case 3) of Proposition \ref{sl3R1dim} we take $k \in \mathbb R$  such 
that 
$n:=k^2 b-2k+c \neq 0$.  Then the element 
$n (e_5-\frac{1}{2} e_8) \in {\bf h}_{31,3}$ is 
conjugate to 
\newline
$-k e_1+k^2(e_2+b(e_5-\frac{1}{2} e_8))+\frac{3k^2b-2k+3c}{2}(e_3-k e_6)+e_4+c(e_5-\frac{1}{2} e_8)+e_7 \in $
\newline
 ${\bf m}_{b,c,d}$ 
 under $g=\left ( \begin{array}{ccc}
0 & \frac{2}{3 n} & \frac{-3k^2b+2b-3c}{3 n} \\
1 & 1 & -k \\
1 & 0 & k \end{array} \right )$. 

\smallskip
\noindent
In the case 4) of Proposition \ref{sl3R1dim}  we take $k \in \mathbb R$ 
 such  that 
 $m:=k^2 b+k+c \neq 0$.  Then the element 
$m (e_5+e_8) \in {\bf h}_{31,4}$ is conjugate to 
\newline
$(3c+3k^2b+k)(ke_2-e_1)+e_4-ke_3+e_7+c(e_5+e_8)+k^2(e_6+b(e_5+e_8)) \in $
\newline  ${\bf m}_{b,c,d}$ 
under $g=\left ( \begin{array}{ccc}
1 & 1 & -k \\
-\frac{1}{3 m} & 0 & \frac{-3c-k-3k^2b}{3 m} \\
0 & 1 & k \end{array} \right )$.   
These facts contradict Lemma \ref{conjugate}. 
\newline
In the  remaining case  one has 
$[{\bf m}_6, {\bf m}_6]={\bf h}_6$ and the loop $L$ with 
 $T_e L={\bf m}_6$ is a Bruck 
loop. 
The assertion follows now from the proof of the Theorem 13 in 
\cite{figula2}, p. 12. 
\end{proof}

\medskip
\noindent
Since the exponential image 
of the Lie algebra  ${\bf g}=\mathfrak{su_3}(\mathbb C,1)$ is much more 
complicated than the exponential image of 
${\bf g}=\mathfrak{sl_3}(\mathbb R)$ we treat the almost differentiable 
left A-loops having $PSU_3(\mathbb C,1)$ as the group topologically 
generated by the left translations under the assumption that their dimension 
is at most 5.

\begin{Prop} Let  $G$ be locally isomorphic to 
$PSU_3(\mathbb C,1)$. Every at most $5$-dimensional connected 
almost differentiable left A-loop having $G$ as the group topologically 
generated by the left translations is isomorphic to the complex hyperbolic 
plane loop $L_0$ having the group 
$Spin _3 \times SO_2(\mathbb R) / \langle (-1,-1) \rangle $ as the 
stabilizer of $e \in L_0$. 
\end{Prop}

\begin{proof} Since the tangent space $T_e L$ for an almost differentiable  
left A-loop $L$ 
is reductive we have to deal  only with  the pairs $({\bf h}, {\bf m})$ 
described in the Propositions \ref{su34dim}, \ref{su33dim}.   
The complex hyperbolic plane loop $L_0$ is realized on the exponential 
image of the subspace ${\bf m}_1$ (cf. \cite{figula2}, p. 8). The Lie 
group corresponding to  ${\bf h}_4$
cannot be the stabilizer of a $4$-dimensional topological  loop $L$ (see 
Lemma 
\ref{fundamenta}). According to ${\bf 1.2}$ the element 
 $e_2 \in {\bf h}_6$  is conjugate to 
$e_1 \in {\bf m}_6$,  
which is a contradiction to Lemma \ref{conjugate}.   
Two loxodromic elements of $\mathfrak{su_3}(\mathbb C,1)$ 
are conjugate in $SU_3(\mathbb C,1)$ if and only if they have the same 
eigenvalues (cf. Prop. 3.2.3 (d) in \cite{chen1}, p. 65) and 
 therefore they are 
conjugate in $SL_3(\mathbb C)$. Since the elements $e_7 \in {\bf h}_7$ and 
$e_4 \in {\bf m}_7$ are loxodromic and $Ad _g(e_7)=e_4$ with 
$g=\left ( \begin{array}{ccc}
0 & 1 & 0 \\
1 & 0 & \sqrt{2} \\
\sqrt{2} & 0 & 1 \end{array} \right ) \in SL_3(\mathbb C)$ we have also a 
contradiction to Lemma \ref{conjugate}. 
\end{proof}

\medskip
\noindent
At the end of this section we show that several reductive spaces 
$({\bf g},{\bf h},{\bf m})$, where ${\bf g}=\mathfrak{su_3}(\mathbb C,1)$ 
and $\hbox{dim}\ {\bf h} \le 2$ can not correspond to an almost 
differentiable left A-loop.

\begin{Prop}
There is no almost differentiable left A-loop 
corresponding to one of  
the following triples:  $({\bf g}, {\bf h}_{12}, {\bf m}_{12})$ in 
Proposition 
\ref{su32dim} and  $({\bf g}, {\bf h}, {\bf m}_a)$ in the case 6) as 
well as  
$({\bf g}, {\bf h}, {\bf m}_b)$ in the case 7) of Proposition \ref{su31dim}. 
\end{Prop} 
\begin{proof}  Since the elements $e_1 \in {\bf h}_{12}$ and $e_2 \in {\bf m}_{12}$ 
are elliptic in a subalgebra isomorphic to  $ \mathfrak{so_3}(\mathbb R)$ 
of ${\bf g}$ (see {\bf 1.2}) they are conjugate under   
$Ad\ PSU_3(\mathbb C,1)$.   
Since he element $e_8 \in {\bf h}$ in the 
case 6 as well as  $e_6+e_7+c e_8 \in {\bf h}$,  $c \neq 0$,  in the case 
7 of Proposition \ref{su31dim} is  hyperbolic in a subalgebra 
isomorphic to 
$ \mathfrak{sl_2}(\mathbb R)$ 
of ${\bf g}$ (see {\bf 1.1}), we have that $e_8$ and $e_7 \in {\bf m}_a$ 
respectively $e_6+e_7+c e_8$ and   
 $-\frac{1}{\sqrt{2+2 c^2}}(e_7-\frac{1}{c} e_8) \in {\bf m}_b$ 
are conjugate  under $Ad\ PSU_3(\mathbb C,1)$. 
This contradicts Lemma  \ref{conjugate}.  
\end{proof}

\section{ Reductive loops corresponding to  semi-simple Lie groups 
of  dimension 6}

\noindent
Let $G=G_1 \times G_2$  be the group topologically generated by the left 
translations of a 
connected almost differentiable left A-loop $L$, such that $G_i$, $i=1,2$,
 is a 
$3$-dimensional quasi-simple Lie group. 
In contrast to the non-existence of $3$-dimensional almost differentiable 
left A-loops belonging to $G$  (cf. Propositions 5 and 8 in 
\cite{figula3})  we will show that there are such loops $L$ with 
$G=G_1 \times G_2$ as the group topologically 
generated by the  left translations if $\hbox{dim} \ L \ge 4$.

\medskip
\noindent
The following fact is well known  from linear algebra:

\begin{Lemma} \label{mdimenzioja} 
Let ${\bf g}={\bf g}_1 \oplus {\bf g}_2$, where ${\bf g}_i$, $i=1,2$ are simple Lie algebras of dimension $3$.  For 
any  subspace ${\bf m}$ with dimension $4$ respectively $5$ the intersections ${\bf m} \cap {\bf g}_1$ and  
${\bf m} \cap {\bf g}_2$ have  dimension at least $1$ respectively at least $2$.  
\end{Lemma}

\medskip
\noindent
The fact that the coset space $G/H$ is parallelizable is reflected in the 
following lemma.  

\begin{Lemma} \label{egyszerso3} Let $G$ be isomorphic to the Lie group  
$G_1 \times G_2$, such that 
$G_2 \cong SO_3(\mathbb R)$ and for the subgroup $H$ of $G$ one has  
$H=H_1 \times H_2$ with $1 \neq H_2 \le G_2$. 
Then $G$ cannot be the group topologically generated by the left 
translations of a topological loop. 
\end{Lemma}

\noindent
For the proof see Lemma 2 in \cite{figula2}, p. 5. 

\medskip
\noindent
First let $G$ be locally isomorphic to 
$SO_3(\mathbb R) \times SO_3(\mathbb R)$.  
Since the at most $2$-dimensional connected subgroups of $G$ are tori and 
$\hbox{dim}\ L \ge 4$ Lemma \ref{fundamenta} gives

\begin{Prop} There is no left A-loop as differentiable section in a 
group locally 
isomorphic to 
$SO_3(\mathbb R) \times SO_3(\mathbb R)$. 
\end{Prop}

\noindent
Now let $G$ be locally isomorphic to 
$PSL_2(\mathbb R) \times G_2$, where $G_2$ is either the group 
$SO_3(\mathbb R)$ or $PSL_2(\mathbb R)$. Using the real basis of 
$\mathfrak{sl_2}(\mathbb R)$ respectively of $\mathfrak{so_3}(\mathbb R)$ 
introduced in {\bf 1.1} respectively in {\bf 1.2} we can choose 
$(e_1,0)$, $(e_2,0)$, 
$(e_3,0)$, $(0,  \varepsilon e_1)$, $(0, \varepsilon  e_2)$, $(0, e_3)$ as 
 a real basis of the Lie algebra 
${\bf g}=\mathfrak{sl_2}(\mathbb R) \oplus {\bf g}_2$, where 
$\varepsilon =\hbox{i}$ with $\hbox{i} ^2=-1$ for  
${\bf g}_2=\mathfrak{so_3}(\mathbb R)$ and $\varepsilon =1$ for  
${\bf g}_2=\mathfrak{sl_2}(\mathbb R)$. 
\newline
Denote by $H$ a subgroup of $G$. 
First  we assume that  $H$ is decomposable into 
a direct product.   
If $H$ has dimension $2$ then  
with Lemma \ref{egyszerso3} we obtain that 
$H$ is (up to interchanging the components) either 
${\mathcal L}_2 \times \{ 1 \}$ or $K_1 \times K_2$, where $K_i$, $i=1,2$ 
are $1$-dimensional 
subgroups of $PSL_2(\mathbb R)$. Now  according to {\bf 1.1} the Lie algebra 
${\bf h}$ of $H$ has one of the following forms:

\medskip
\noindent
\centerline{${\bf h}_1=\langle (e_3,0), (0, e_3) \rangle $, \ \ 
${\bf h}_2=\langle (e_3,0), (0, e_2+e_3) \rangle $, \ \  
${\bf h}_3=\langle (e_3,0), (0, e_1) \rangle $, }

\medskip
\noindent
\centerline{${\bf h}_4=\langle (e_1,0), (0, e_1) \rangle $, \ \ 
${\bf h}_5=\langle (e_1,0), (0, e_2+e_3) \rangle $, }

\medskip
\noindent
\centerline{${\bf h}_6=\langle (e_2+e_3,0), (0, e_2+e_3) \rangle $, \ \ 
${\bf h}_7=\langle (e_1,0), (e_2+e_3,0) \rangle $.} 

\medskip
\noindent 
The Lie algebras ${\bf h}_1$ till ${\bf h}_7$ are subalgebras of 
${\bf g}=\mathfrak{sl_2}(\mathbb R) \oplus \mathfrak{sl_2}(\mathbb R)$ 
but  ${\bf h}_7$ is also a subalgebra of 
${\bf g}=\mathfrak{sl_2}(\mathbb R) \oplus \mathfrak{so_3}(\mathbb R)$. 
\newline
If $\hbox{dim}\ H=1$ then $H$ has the shape $K_1 \times \{ 1 \}$ with a 
$1$-dimensional 
subgroup $K_1$  of $PSL_2(\mathbb R)$. 
Then according to {\bf 1.1} the Lie algebra ${\bf h}$ of 
$H$ has (up to interchanging the components) one of the following forms:

\medskip
\noindent
\centerline{${\bf h}_8=\langle (e_3,0)  \rangle $, \ \ 
${\bf h}_9=\langle (e_1,0)  \rangle $, \ \  
${\bf h}_{10}=\langle (e_2+e_3,0)  \rangle $. }

\medskip
\noindent
These algebras are subalgebras of 
${\bf g}=\mathfrak{sl_2}(\mathbb R) \oplus \mathfrak{sl_2}(\mathbb R)$ 
as well as 
${\bf g}=\mathfrak{sl_2}(\mathbb R) \oplus \mathfrak{so_3}(\mathbb R)$. 
\newline
Now we suppose that  $H$ is not a  direct product of two subgroups.  In 
the case  
$\hbox{dim}\ H=2$ one has 
$H=\{ (x, \varphi (x))|\ x \in {\mathcal L}_2 \}$,  where  
$\varphi \neq 1$ is a homomorphism of ${\mathcal L}_2$
into  $PSL_2(\mathbb R)$. If $\varphi $ is injective  then the Lie algebra 
of $H$ is a subalgebra of 
$\mathfrak{sl_2}(\mathbb R) \oplus \mathfrak{sl_2}(\mathbb R)$ and 
has the shape 

\medskip
\noindent
\centerline{${\bf h}_{11}= \langle (e_1,e_1), (e_2+e_3, e_2+e_3) \rangle $.}
 
\medskip
\noindent
If $\varphi $ has $1$-dimensional kernel  then the Lie algebra of $H$ is 
given by 

\medskip
\noindent
\centerline{${\bf h}_{12}= \langle (e_1,k), (e_2+e_3, 0) \rangle $,}  

\medskip
\noindent
where $k$ denotes 
either the element $e_1$ or $e_2+e_3$ of $\mathfrak{sl_2}(\mathbb R)$ 
or $e_3$ of $\mathfrak{sl_2}(\mathbb R) \cap \mathfrak{so_3}(\mathbb R) $ 
(see {\bf 1.1} and {\bf 1.2}).   

\medskip
\noindent 
In the case  $\hbox{dim}\ H=1$ one has 
$H=\{ (k_1, \varphi(k_1))| \ k_1 \in K_1\}$, where  $K_1$ is a 
$1$-dimensional subgroup of 
$PSL_2(\mathbb R)$ 
and $\varphi \neq 1$ is a homomorphism of $K_1$ 
into  $PSL_2(\mathbb R)$ or $SO_3(\mathbb R)$.
Then the Lie algebra ${\bf h}$  of  $H$  has (up to interchanging the 
components) one of the 
following forms:  

\medskip
\noindent
\centerline{${\bf h}_{13}=\langle (e_1,e_1) \rangle $, \ \  
${\bf h}_{14}=\langle (e_1,e_2+e_3) \rangle $, \ \  
${\bf h}_{15}=\langle (e_2+e_3,e_2+e_3) \rangle $, }

\medskip
\noindent
\centerline{${\bf h}_{16}=\langle (e_1,e_3) \rangle $, \ \ 
${\bf h}_{17}=\langle (e_2+e_3,e_3) \rangle $, \ \  
${\bf h}_{18}=\langle (e_3,e_3) \rangle $. }

\medskip
\noindent
The Lie algebra ${\bf h}_{13}$ till ${\bf h}_{18}$ are subalgebras 
of ${\bf g}=\mathfrak{sl_2}(\mathbb R) \oplus \mathfrak{sl_2}(\mathbb R)$ 
but 
 ${\bf h}_{16}, {\bf h}_{17}, {\bf h}_{18}$ are also 
subalgebras 
of  
${\bf g}=\mathfrak{sl_2}(\mathbb R) \oplus \mathfrak{so_3}(\mathbb R)$. 

\begin{Prop} \label{reductiveslg2} The Lie algebra 
${\bf g}=\mathfrak{sl_2}(\mathbb R) 
\oplus {\bf g}_2$, where ${\bf g}_2$ is a $3$-dimensional simple Lie 
algebra,  
 is reductive with an at most $2$-dimensional subalgebra 
${\bf h}$ and a complementary subspace ${\bf m}$ generating ${\bf g}$ in 
exactly one of the 
following cases:  

\medskip
\noindent
1) ${\bf h}_8=\langle (e_3,0) \rangle $, ${\bf m}_a=\langle (e_1,0), (e_2,0), (0, \varepsilon e_1), (0, \varepsilon e_2), (a e_3, e_3) \rangle $,

\medskip
\noindent
2) ${\bf h}_8=\langle (e_3,0) \rangle $, ${\bf m}_b=\langle (e_1,0), (e_2,0), (0, \varepsilon e_1), (b e_3,  \varepsilon e_2), (0, e_3) \rangle $,

\medskip
\noindent
3) ${\bf h}_8=\langle (e_3,0) \rangle $, ${\bf m}_c=\langle (e_1,0), (e_2,0), (c e_3,  \varepsilon e_1), (0,  \varepsilon e_2), (0, e_3) \rangle $,

\medskip
\noindent
4) ${\bf h}_9=\langle (e_1,0) \rangle $, ${\bf m}_d=\langle (e_2,0), 
(e_3,0), (0,  \varepsilon e_1), (0,  \varepsilon e_2), (d e_1, e_3) \rangle 
$,  

\medskip
\noindent
5) ${\bf h}_9=\langle (e_1,0) \rangle $, ${\bf m}_f=\langle (e_2,0), (e_3,0), (0,  \varepsilon e_1), (f e_1,  \varepsilon e_2), (0, e_3) 
\rangle $, 

\medskip
\noindent
6) ${\bf h}_9=\langle (e_1,0) \rangle $, ${\bf m}_g=\langle (e_2,0), (e_3,0), (g e_1,  \varepsilon e_1), (0,  \varepsilon e_2), (0, e_3) 
\rangle $, 

\medskip
\noindent   
7) ${\bf h}_{16}=\langle (e_1,e_3) \rangle $, ${\bf m}_h=\langle (e_2,0), 
(e_3,0), (0, \varepsilon e_1), (0, \varepsilon e_2), (h e_1, (1+h) e_3) 
\rangle $,

\medskip
\noindent
8) ${\bf h}_{17}=\langle (e_2+e_3,e_3) \rangle $, 
${\bf m}_k=\langle (e_3,k e_3), 
(e_1,0), (0, \varepsilon e_1), (0, \varepsilon e_2), (e_2+e_3, 0) 
\rangle $, 

\medskip
\noindent
9) ${\bf h}_{18}=\langle (e_3,e_3) \rangle $, 
${\bf m}_l=\langle (l e_3,(1+l) e_3), 
(e_1,0), (e_2,0),  (0, \varepsilon e_1), 
(0, \varepsilon e_2) \rangle $,  

\medskip
\noindent
10) ${\bf h}_1=\langle (e_3,0), (0, e_3) \rangle $, 
${\bf m}_1=\langle (e_1,0), (e_2,0), (0, e_1), (0, e_2) \rangle $, 

\medskip
\noindent
11) ${\bf h}_3=\langle (e_3,0), (0, e_1) \rangle $, 
${\bf m}_3=\langle (e_1,0), (e_2,0), (0, e_2), (0, e_3) \rangle $, 

\medskip
\noindent
12) ${\bf h}_4=\langle (e_1,0), (0, e_1) \rangle $, 
${\bf m}_4=\langle (e_2,0), (e_3,0), (0, e_2), (0, e_3) \rangle $, 

\medskip
\noindent
13) ${\bf h}_{13}=\langle (e_1,e_1) \rangle $, 
${\bf m}_m=\langle  (e_2,0), (e_3,0), 
(0,e_3), (0,e_2), (m e_1,(1+m)e_1) \rangle$,

\medskip
\noindent
14) ${\bf h}_{14}=\langle (e_1,e_2+e_3) \rangle $, 
${\bf m}_n=\langle (e_2,0), (e_3,0), 
(0,e_1), (0,e_2+e_3), (n e_1,e_2) \rangle $, 

\medskip
\noindent
where $a,b,c,d,f,g,h,k,l,m,n \in \mathbb R$ and 
$\varepsilon =$ $\hbox{i}$ for  ${\bf g}_2=\mathfrak{so_3}(\mathbb R)$ 
whereas 
$\varepsilon =1$ for  ${\bf g}_2=\mathfrak{sl_2}(\mathbb R)$. The cases 
1) till 10) occur for both simple $3$-dimensional Lie algebras whereas 
 the cases 10) till 14) occur only for 
${\bf g}_2=\mathfrak{sl_2}(\mathbb R)$.   
\end{Prop}
\begin{proof}
The basis elements of an arbitrary complement ${\bf m}_i$ to ${\bf h}_i$, 
$i=1, \cdots , 18$, in 
${\bf g}=\mathfrak{sl_2}(\mathbb R) \oplus {\bf g}_2$, where 
${\bf g}_2$ is either $\mathfrak{sl_2}(\mathbb R)$ or 
$\mathfrak{so_3}(\mathbb R)$,  are:  
\newline
In the case $i=1$ 

\medskip
\noindent 
\centerline{$(e_1+a_1 e_3, a_2 e_3)$, \ \ $(e_2+b_1 e_3, b_2 e_3)$, \ \ $(c_1 e_3, e_1+c_2 e_3)$, \ \ $(d_1 e_3, e_2+d_2 e_3)$, }

\medskip
\noindent 
in the case $i=2$ 

\medskip
\noindent 
\centerline{$(e_1+a_1 e_3, a_2 (e_2+ e_3))$, \ \ 
$(e_2+b_1 e_3, b_2 (e_2+e_3))$, }

\medskip
\noindent 
\centerline{$(c_1 e_3, e_1+c_2 (e_2+e_3))$, \ \ 
$(d_1 e_3, e_3+d_2 (e_2+e_3))$, }
 
\medskip
\noindent 
in the case $i=3$ 

\medskip
\noindent 
\centerline{$(e_1+a_1 e_3, a_2 e_1)$, \ \ $(e_2+b_1 e_3, b_2 e_1)$, \ \ 
$(c_1 e_3, e_2+c_2 e_1)$, \ \ $(d_1 e_3, e_3+d_2 e_1)$, }
 
\medskip
\noindent 
in the case $i=4$ 

\medskip
\noindent 
\centerline{$(e_2+a_1 e_1, a_2 e_1)$, \ \ $(e_3+b_1 e_1, b_2 e_1)$, \ \ 
$(c_1 e_1, e_2+c_2 e_1)$, \ \ $(d_1 e_1, e_3+d_2 e_1)$, }
 
\medskip
\noindent 
in the case $i=5$ 

\medskip
\noindent 
\centerline{$(e_2+a_1 e_1, a_2 (e_2+e_3))$, \ \ $(e_3+b_1 e_1, b_2 (e_2+e_3))$, } 

\medskip
\noindent 
\centerline{$(c_1 e_1, e_1+c_2 (e_2+e_3))$, \ \ $(d_1 e_1, e_3+d_2 (e_2+e_3))$, }
 
\medskip
\noindent 
in the case $i=6$ 

\medskip
\noindent 
\centerline{$(e_1+a_1 (e_2+e_3), a_2 (e_2+e_3))$, \ \ $(e_3+b_1 (e_2+e_3), b_2 (e_2+e_3))$, } 

\medskip
\noindent 
\centerline{$(c_1 (e_2+e_3), e_1+c_2 (e_2+e_3))$, \ \ $(d_1 (e_2+e_3), e_3+d_2 (e_2+e_3))$, }

\medskip
\noindent 
in the case $i=7$ 

\medskip
\noindent 
\centerline{$(e_3+a_1 e_1 +a_2 (e_2+e_3),0)$, \ \ $(b_1 e_1+ b_2 (e_2+e_3), 
\varepsilon e_1)$, } 

\medskip
\noindent 
\centerline{$(c_1 e_1+c_2 (e_2+e_3), \varepsilon e_2)$, \ \ 
$(d_1 e_1+d_2 (e_2+e_3), e_3)$, }

\medskip
\noindent
in the case $i=8$ 

\medskip
\noindent
\centerline{$(e_1+a_1 e_3,0)$, \ \ $(e_2+a_2 e_3,0)$, \ \ 
$(a_3 e_3, \varepsilon e_1)$, \ \ $(a_4 e_3, \varepsilon e_2)$, \ \  
$(a_5 e_3,e_3)$, }   

\medskip
\noindent
in the case $i=9$ 

\medskip
\noindent
\centerline{$(e_2+a_1 e_1,0)$, \ \ $(e_3+a_2 e_1,0)$, \ \ 
$(a_3 e_1, \varepsilon e_1)$, \ \ $(a_4 e_1, \varepsilon e_2)$, \ \  
$(a_5 e_1,e_3)$, }

\medskip
\noindent
in the case $i=10$ 

\medskip
\noindent
\centerline{$(e_2+a_1 (e_2+e_3),0)$, \ \ $(e_1+a_2 (e_2+e_3),0)$, \ \ 
$(a_3 (e_2+e_3), \varepsilon e_1)$, }

\medskip
\noindent
\centerline{$(a_4 (e_2+e_3), \varepsilon e_2)$, \ \  $(a_5 (e_2+e_3), e_3)$, }

\medskip
\noindent
in the case $i=11$ 

\medskip
\noindent
\centerline{$(e_3+a_1 e_1+a_2(e_2+e_3),a_1 e_1+a_2(e_2+e_3))$, }

\medskip
\noindent
\centerline{$(b_1 e_1+b_2(e_2+e_3),e_1+b_1 e_1+b_2(e_2+e_3))$, }

\medskip
\noindent
\centerline{$(c_1 e_1+c_2(e_2+e_3),e_2+c_1 e_1+c_2(e_2+e_3))$, }  

\medskip
\noindent
\centerline{$(d_1 e_1+d_2(e_2+e_3),e_3+d_1 e_1+d_2(e_2+e_3))$, }

\medskip
\noindent
in the case $i=12$  

\medskip
\noindent
\centerline{$(e_3+a_1 e_1+a_2(e_2+e_3),a_1 k)$, \  
$(b_1 e_1+b_2(e_2+e_3), \varepsilon e_1+b_1 k)$, }

\medskip
\noindent
\centerline{$(c_1 e_1+c_2(e_2+e_3), \varepsilon e_2+c_1 k)$, \  
$(d_1 e_1+d_2(e_2+e_3),e_3+d_1 k)$, }

\medskip
\noindent
in the case $i=13$ 

\medskip
\noindent
\centerline{$(e_2+a_1 e_1,a_1 e_1)$, \ \ $(e_3+a_2 e_1,a_2 e_1)$, \ \ 
$(a_3 e_1, e_1+a_3 e_1)$, }

\medskip
\noindent
\centerline{$(a_4 e_1, e_2+a_4 e_1)$, \ \  $(a_5 e_1,e_3+a_5 e_1)$, }

\medskip
\noindent
in the case $i=14$ 

\medskip
\noindent
\centerline{$(e_2+a_1 e_1,a_1 (e_2+e_3))$, \ \ $(e_3+a_2 e_1,a_2 (e_2+e_3))$, \ \ 
$(a_3 e_1, e_1+a_3 (e_2+e_3))$, }

\medskip
\noindent
\centerline{$(a_4 e_1, e_2+a_4 (e_2+e_3))$, \ \  $(a_5 e_1,e_3+a_5 (e_2+e_3))$, }

\medskip
\noindent
in the case $i=15$ 

\medskip
\noindent
\centerline{$(e_2+a_1 (e_2+e_3),a_1 (e_2+e_3))$, \ \ $(e_1+a_2 (e_2+e_3),a_2 (e_2+e_3))$, }

\medskip
\noindent
\centerline{$(a_3 (e_2+e_3), e_1+a_3 (e_2+e_3))$, \ \ $(a_4 (e_2+e_3), e_2+a_4 (e_2+e_3))$, }

\medskip
\noindent
\centerline{$(a_5 (e_2+e_3),e_3+a_5 (e_2+e_3))$, }

\medskip
\noindent
in the case $i=16$ 

\medskip
\noindent
\centerline{$(e_2+a_1 e_1,a_1 e_3)$, \ \ $(e_3+a_2 e_1,a_2 e_3)$, \ \ 
$(a_3 e_1, \varepsilon e_1+a_3 e_3)$, }

\medskip
\noindent
\centerline{$(a_4 e_1, \varepsilon e_2+a_4 e_3)$, \ \  
$(a_5 e_1,e_3+a_5 e_3)$, }

\medskip
\noindent
in the case $i=17$

\medskip
\noindent
\centerline{$(e_2+a_1 (e_2+e_3),a_1 e_3)$, \ \ $(e_1+a_2 (e_2+e_3),a_2 e_3)$, \ \ 
$(a_3 (e_2+e_3), \varepsilon e_1+a_3 e_3)$, }

\medskip
\noindent
\centerline{$(a_4 (e_2+e_3), \varepsilon e_2+a_4 e_3)$, \ \  
$(a_5 (e_2+e_3),e_3+a_5 e_3)$, }

\medskip
\noindent
in the case $i=18$ 

\medskip
\noindent
\centerline{$(e_1+a_1 e_3,a_1 e_3)$, \ \ $(e_2+a_2 e_3,a_2 e_3)$, \ \ 
$(a_3 e_3, \varepsilon e_1+a_3 e_3)$, } 

\medskip
\noindent
\centerline{$(a_4 e_3, \varepsilon e_2+a_4 e_3)$, \ \  $(a_5 e_3,e_3+a_5 e_3)$, } 

\medskip
\noindent
where $a_i$, $i=1,2, \cdots ,5$, $b_j$, $j=1,2$, $c_k$, $k=1,2$, $d_l$, 
$l=1,2$, are real parameters, $\varepsilon =\hbox{i}$ for 
${\bf g}_2=\mathfrak{so_3}(\mathbb R)$ and $\varepsilon =1$ for 
${\bf g}_2=\mathfrak{sl_2}(\mathbb R)$.  
\newline 
Using the relation $[{\bf h}_i,{\bf m}_i] \subseteq {\bf m}_i$, 
$i=1, \cdots ,18$,  and   Lemma \ref{mdimenzioja} we obtain the assertion. 
\end{proof}

\begin{Prop} Let $G$ be locally isomorphic to 
$PSL_2(\mathbb R) \times G_2$, where $G_2$ is either $PSL_2(\mathbb R)$ or 
$SO_3(\mathbb R)$. If $G$ is the group topologically generated by the 
 left translations of a connected almost differentiable  left A-loop $L$ 
then $L$ is either  a Scheerer extension of 
$G_2$ by $\mathbb H_2$ or  the direct product  
$\mathbb H_2 \times \mathbb H_2$, where $\mathbb H_2$ denotes 
 the hyperbolic plane loop. 
In the second case $G$ is isomorphic to 
$PSL_2(\mathbb R) \times PSL_2(\mathbb R)$.  
\end{Prop}
\begin{proof} Since we assume that 
$\hbox{dim}\ L \ge 4$ 
we have to consider only the pairs $({\bf h}, {\bf m})$ in 
Proposition  \ref{reductiveslg2}. 
Now using  {\bf 1.1} and {\bf 1.2}  we obtain that 
 the element $(0, e_1) \in {\bf h}_3 \cap {\bf h}_4$, the 
element $(e_1,0) \in {\bf h}_9$, the element $(e_1,e_1) \in {\bf h}_{13}$ 
respectively 
the element $(e_1,e_2+e_3) \in {\bf h}_{14}$ is conjugate in this order to 
$(0,e_2) \in {\bf m}_3 \cap {\bf m}_4$, to 
$(e_2,0) \in {\bf m}_d \cap {\bf m}_f \cap {\bf m}_g$, to 
$(e_2,e_2) \in {\bf m}_{m}$ 
respectively to $(e_2,e_2+e_3) \in {\bf m}_{n}$. Hence there exists no 
global left A-loop $L$ 
such that $T_e L$ is a reductive complement listed in the cases 4), 5), 6), 
 11), 12), 13), 14)  
 (see. Lemma \ref{conjugate}). 

\medskip
\noindent
Now we consider the reductive complements ${\bf m}_a$, ${\bf m}_b$, 
${\bf m}_c$ in 1) till 3) of  Proposition \ref{reductiveslg2}. 
First we assume that  $a \neq 0$, $b \neq 0$, $c \neq 0$. 
The vectors $v_{j,l}=\left( k e_3,  \frac{k}{l} \varepsilon e_j \right)$, 
$w_{j,l}=\left( \sqrt{ k^2 -4 \pi ^2 } e_2+ k e_3, \frac{k}{l} \varepsilon 
e_j \right)$, where $k > 2 \pi$ is an integer,  
are contained in the subspace ${\bf m}_a$ for $j=3$, $l=a$ and 
 $\varepsilon=1$,
in the  
subspace ${\bf m}_b$ for $j=2$, $l=b$, respectively in 
${\bf m}_c$ for $j=1$, $l=c$, where 
 $\varepsilon=1$ for 
${\bf g}_2= \mathfrak{sl_2}(\mathbb R)$ and  $\varepsilon=\hbox{i}$ for 
${\bf g}_2= \mathfrak{so_3}(\mathbb R)$.   
According to {\bf 1.1} and {\bf 1.2} the images of $v_{j,l},  w_{j,l}$, 
$j=1, 2, 3$,  under the exponential map have the following 
representatives in $PSL_2(\mathbb R) \times G_2$:   

\medskip
\noindent
\centerline{$m_1=\exp {v_{3,a}}=
\left (  A, 
 \left ( \begin{array}{rr}
\cos { \frac{k}{a}} & \sin { \frac{k}{a}} \\
-\sin { \frac{k}{a}} & \cos { \frac{k}{a}} \end{array} \right ) \right )$, } 
 
\medskip
\noindent
\centerline{$m_2=\exp {w_{3,a}}=\left ( I , \left (\begin{array}{rr}
\cos { \frac{k}{a}} & \sin { \frac{k}{a} } \\
-\sin { \frac{k}{a} } & \cos { \frac{k}{a}} \end{array} \right ) \right )$, } 
 
\medskip
\noindent
\centerline{$m_3=\exp {v_{2,b}}=
\left ( A, 
 \left ( \begin{array}{rr}
\cosh { (\frac{k}{b} \varepsilon ) } &  \sinh {(\frac{k}{b} \varepsilon )} \\
- \sinh {(\frac{k}{b} \varepsilon )} & \cosh {(\frac{k}{b} \varepsilon )} 
\end{array} 
\right ) \right )$, } 

\medskip
\noindent
\centerline{$m_4=\exp {w_{2,b}}=
\left ( \pm I,  
 \left ( \begin{array}{rr}
\cosh {(\frac{k}{b} \varepsilon )} & \sinh {(\frac{k}{b} \varepsilon )} \\
- \sinh {(\frac{k}{b} \varepsilon ) } & \cosh {(\frac{k}{b} \varepsilon )} 
\end{array} \right ) \right )$, } 
 
\medskip
\noindent
\centerline{$m_5=\exp {v_{1,c}}=\left (  A, 
 \left ( \begin{array}{cc}
\cosh {(\frac{k}{c} \varepsilon )}+ \sinh {(\frac{k}{c} \varepsilon )} & 0  \\
0  & \cosh {(\frac{k}{c} \varepsilon )}- \sinh {(\frac{k}{c} 
\varepsilon ) } \end{array} \right ) \right )$, }

\medskip
\noindent
\centerline{$m_6=\exp {w_{1,c}}=\left ( I,  
 \left ( \begin{array}{cc}
\cosh {(\frac{k}{c} \varepsilon )}+ \sinh {(\frac{k}{c} \varepsilon )}  &  0 \\
0 & \cosh {(\frac{k}{c} \varepsilon )}- \sinh {(\frac{k}{c} 
\varepsilon )} \end{array} \right ) \right )$, } 

\medskip
\noindent
where $A=\left ( \begin{array}{rr}
\cos {k} & \sin {k} \\
-\sin {k} & \cos {k} \end{array} \right )$, $\varepsilon =$ $\hbox{i}$ for 
${\bf g}_2=\mathfrak{so_3}(\mathbb R)$,   whereas  
$\varepsilon =1$ for ${\bf g}_2=\mathfrak{sl_2}(\mathbb R)$.    
For the representatives 

\medskip
\noindent
\centerline{$g_1=\left ( I , \left (\begin{array}{rr}
\cos { \frac{k}{a}} & \sin { \frac{k}{a} } \\
-\sin { \frac{k}{a} } & \cos { \frac{k}{a}} \end{array} \right ) \right )$, }

\medskip
\noindent
\centerline{$g_2=\left ( I,  \left ( \begin{array}{rr}
\cosh {(\frac{k}{b} \varepsilon )} & \sinh {(\frac{k}{b} \varepsilon )} \\
- \sinh { (\frac{k}{b} \varepsilon )} & \cosh { (\frac{k}{b} 
\varepsilon )} \end{array} \right ) \right )$, }

\medskip
\noindent
\centerline{$g_3=\left ( I,  
 \left ( \begin{array}{cc}
\cosh { (\frac{k}{c} \varepsilon )}+ 
\sinh {(\frac{k}{c} \varepsilon )}  &  0 \\
0 & \cosh { (\frac{k}{c} \varepsilon )}-  
\sinh { (\frac{k}{c} \varepsilon )}\end{array} \right ) \right )$ }

\medskip
\noindent
we have  $g_1=m_1 \cdot h_1=m_2$,  $g_2=m_3 \cdot h_1=m_4$, 
$g_3=m_5 \cdot h_1=m_6$ such that 
$h_1= \left ( A^{-1}, I \right )$. 
These facts  again contradict Lemma \ref{conjugate}.

\medskip
\noindent
For $a=0$, $b=0$, $c=0$ the 
complements ${\bf m}_a$, ${\bf m}_b$, ${\bf m}_c$ in 1) till 3) of  
Proposition \ref{reductiveslg2} reduce to 
${\bf m}_0=\langle (e_1,0), (e_2,0), (0, \varepsilon  e_1), 
(0, \varepsilon e_2), 
(0, e_3) \rangle $. The exponential image $\exp {\bf m}_0$ is the direct 
product  $M \times G_2$, such that $M$ is the image of the 
section corresponding to the hyperbolic plane loop $\mathbb H_2$ (cf. 
\cite{loops}, pp. 283-284) and $G_2$ is the group $PSL_2(\mathbb R)$ 
respectively $SO_3(\mathbb R)$ according whether 
 $\varepsilon =1$ or $\varepsilon =$ $\hbox{i} $.  Since 
$H$ has the shape 
$H_1 \times \{ 1 \}$, where 
$H_1 \cong SO_2(\mathbb R) \le PSL_2(\mathbb R)$ the  
global loop $L_0$ realized on  $\exp {\bf m}_0$ is the direct product of 
$\mathbb H_2$  and $G_2$.

\medskip
\noindent
Now we treat the complements ${\bf m}_h$,  ${\bf m}_k$, ${\bf m}_l$, 
$h,k,l \in \mathbb R$ of the cases 7) till 9) in Proposition 
\ref{reductiveslg2}. The 
reductive complement 
${\bf m}_a, a \in \mathbb R$, ${\bf m}_b, b \in \mathbb R$, respectively   
${\bf m}_c, c \in \mathbb R$ of Lemma 12 in 
\cite{figula3}, p. 404, is in this order a subspace of 
${\bf m}_h$,  ${\bf m}_k$, respectively  
${\bf m}_l$.  Moreover, the subalgebra ${\bf h}_{16}$ in the case 7) 
coincides with the subalgebra ${\bf h}$ in case 1) of Lemma 12 in 
\cite{figula3},  the subalgebra ${\bf h}_{17}$ in the case 8) 
is equal  with the subalgebra ${\bf h}$ in case 2) of Lemma 12 in 
\cite{figula3}, and the subalgebra ${\bf h}_{18}$ in the case 9) 
coincides with the subalgebra ${\bf h}$ in case 3) of Lemma 12 in 
\cite{figula3}, p. 404. Hence the same computations as in the proof of 
Proposition 13 in 
\cite{figula3}, pp. 404-406, show that for $h \neq -1$ the complement 
${\bf m}_h$, for $k \neq 0$ the complement 
${\bf m}_k$ and  
for $l \notin \{ 0, -1 \}$ the complement ${\bf m}_l$ cannot be the 
tangent space of a global almost differentiable left A-loop. 

\smallskip
\noindent
It remains to consider the complements ${\bf m}_{h=-1}$, 
${\bf m}_{k=0}$, ${\bf m}_{l=0}$ and  ${\bf m}_{l=-1}$. 
First let $\varepsilon =$ $\hbox{i}$. Then 
the element 
$(e_1,e_3) \in {\bf h}_{16}$ is conjugate to 
$(e_2,\hbox{i} e_1) \in {\bf m}_{h=-1}$, the element 
$(e_2+e_3,e_3) \in {\bf h}_{17}$ is conjugate to 
$(e_2+e_3,\hbox{i} e_1) \in {\bf m}_{k=0}$ and the element 
$(e_3,e_3) \in {\bf h}_{18}$ is 
conjugate to $(e_3,\hbox{i} e_1) \in {\bf m}_{l=-1}$ (see {\bf 1.2}), 
which are contradictions to Lemma \ref{conjugate}.  
Since the exponential image of the Lie algebra ${\bf h}_{18}$ 
has the shape $H_n=\{ (x, x^n) \ | \ x \in SO_2(\mathbb R), 
n \in \mathbb N \backslash \{ 0 \} \}$ the exponential image 
$M \times SO_3(\mathbb R)$ of the complement ${\bf m}_{l=0}$, where $M$ 
is the image of the section belonging to the hyperbolic plane loop 
$\mathbb H_2$ (cf. \cite{loops}, pp. 283-284), yields 
Scheerer extensions of $SO_3(\mathbb R)$ by $\mathbb H_2$ (cf. \cite{loops}, 
Section 2). 

\smallskip
\noindent
Finally let $\varepsilon =1$. The  complements ${\bf m}_{h=-1}$, 
${\bf m}_{k=0}$,  
${\bf m}_{l=-1}$ and ${\bf m}_{l=0}$ 
are (up to interchanging the components) equal to the vector space 

\medskip
\noindent
\centerline{${\bf m}'=\langle (e_1,0), (e_2,0), (e_3,0), 
(0, e_1), (0, e_2) \rangle $}

\medskip
\noindent
and its 
exponential image $\exp {\bf m}'$ is 
the direct 
product  $PSL_2(\mathbb R) \times M$, where  $M$ is the image of the 
section corresponding to  $\mathbb H_2$. 
The group $H=\{ (\varphi (x),x) \ | \ x \in SO_2(\mathbb R) \}$ coincides 
with the group $H_{16}$ belonging to  ${\bf h}_{16}$ 
respectively with $H_{17}$ of ${\bf h}_{17}$ if 
 $\varphi $ is a homomorphism from $SO_2(\mathbb R)$  onto a hyperbolic 
respectively  a 
parabolic $1$-parameter subgroup of $PSL_2(\mathbb R)$. 
The subgroup  $H_{18}$ of 
${\bf h}_{18}$ has the form: 
$H'_n=\{ (x^n, x) \ | \ x \in SO_2(\mathbb R), 
n \in \mathbb N \backslash \{ 0 \} \}$. According to \cite{loops}, 
Section 2, any loop $L$ realized  on the factor 
space $G/H_n$, $n=16, 17, 18,$ and having $\exp {\bf m}'$ as the image of 
its section is a  Scheerer extension of the Lie 
group $PSL_2(\mathbb R)$ by  $\mathbb H_2$.  

\smallskip
\noindent
All  Scheerer 
extensions having $PSL_2(\mathbb R) \times SO_3(\mathbb R)$ or 
$PSL_2(\mathbb R) \times PSL_2(\mathbb R)$ as the group topologically 
generated by their left translations
satisfy the Bol 
identity because of $\big [ [{\bf m}, {\bf m}], {\bf m} \big ] \subset 
{\bf m}$ but they  are not Bruck loops since there is no involutory 
automorphism $\sigma :{\bf g} \to {\bf g}$ such that 
$\sigma ({\bf m})=-{\bf m}$ and $\sigma ({\bf h})={\bf h}$.

\medskip
\noindent 
In the remaining  case 10) in Proposition \ref{reductiveslg2} 
the subgroup $H_1$ of ${\bf h}_1$ is the direct product 
$SO_2(\mathbb R) \times SO_2(\mathbb R)$ and the exponential image 
$M_1$ of ${\bf m}_1$ is the 
direct product $M \times M$, where $M$ is the image of the section 
belonging to  
$\mathbb H_2$. 
According to Proposition 1.19 in \cite{loops}, p. 28, the loop $L$ is 
the direct product 
$\mathbb H_2 \times \mathbb H_2$.  
\end{proof}

\medskip
\noindent
Eingegangen am 2. August 2006. 

\medskip
\noindent 
Author address: \\
Mathematisches Institut der Universit\"at 
Erlangen-N\"urnberg, \\   Bismarckstr. 1 $\frac{1}{2}$, \\  
D-91054 Erlangen, Germany; \\ 
and Institute of Mathematics, University of Debrecen, \\ 
 P.O.B. 12, H-4010 Debrecen,  Hungary \\
 figula@mi.uni-erlangen.de;  figula@math.klte.hu


\begin{thebibliography}{37}

\bibitem{betten} D. Betten, \textit{Die komplex-hyperbolische Ebene}, Math. Z. {\bf 132}, (1973), 249-259. 


\bibitem{chen} S. S. Chen, \textit{On subgroups of the Noncompact real 
exceptional Lie group $F^*_4$}, Math. Ann. {\bf 204} (1973), 271-284.  

\bibitem{chen1} S. S. Chen and L. Greenberg, \textit{Hyperbolic spaces}, pp. 
49-81, In: Contributions to analysis, ed: L. V. Ahlfors, I. Kra, 
B. Maskit, L. Nirenberg, Academic Press, New York-London, 1974.   

\bibitem{atri} J. E. D'Atri and W. Ziller, \textit{Naturally reductive 
metrics and Einstein metrics on compact Lie groups}, Mem. Amer. Math. Soc. 
{\bf 215} (1979). 



\bibitem{figula2} \'A. Figula, \textit{Bol loops as sections in semi-simple 
Lie groups of small dimension}, to appear in Manuscr. Math. 

\bibitem{figula3} \'A. Figula, \textit{3-dimensional loops on 
non-solvable reductive spaces}, 
Adv. Geom. {\bf 5} (2005), 391-420. 

\bibitem{freudenthal} H. Freudenthal and H. de Vries, \textit{Linear 
Lie Groups}, Academic 
Press, New York-London, 1969. 



\bibitem{helgason} S. Helgason, \textit{Differential Geometry and Symmetric 
Spaces}, Academic Press, New York, 1962. 

\bibitem{hilgert} J. Hilgert and K. H. Hofmann, \textit{Old and new on 
$SL(2)$}, Manuscr. Math. {\bf 54} (1985), 17-52. 



\bibitem{kikkawa1} M. Kikkawa, \textit{On locally reductive spaces and tangent algebras}, Mem. Fac. Lit. Sci. Shimane Univ. Natur. Sci. {\bf 5} (1972), 
1-13. 

\bibitem{kikkawa2} M. Kikkawa, \textit{Geometry of homogeneous Lie loops}, 
Hiroshima Math. J. {\bf 5} (1975), 141-179.

\bibitem{kobayashi} S. Kobayashi and K. Nomizu, \textit{Foundations of 
Differential Geometry  II}, Interscience Publishers, New York-London-Sydney, 
 1969.     


\bibitem{kowalski} O. Kowalski and L. Vanhecke, \textit{Classification of 
five-dimensional naturally reductive spaces}, Math. Proc. Cambridge 
Philos. Soc. {\bf 97} (1985), 445-463. 


\bibitem{lie2} S. Lie and F. Engel, \textit{Theorie der Transformationsgruppen}, 3. Abschnitt, Verlag von B.G. Teubner, Leipzig, 1893. 


\bibitem{lie} S. Lie, \textit{Vorlesungen \"uber 
Continuierliche Gruppen}, Bearbeitet und herausgegeben von  G. Scheffers, 
Chelsea Publishing Company, Bronx, New York, 1971.  


\bibitem{loos} O. Loos, \textit{Symmetric Spaces}, Vol I, 
Benjamin, New York 1969. 



\bibitem{quasigroups} P. O. Miheev and L. V. Sabinin, \textit{Quasigroups
and Differential Geometry, Chapter XII in Quasigroups and Loops:Theory and
Applications (O. Chein, H.O. Pflugfelder and J.D.H. Smith), Sigma Series
in Pure Math. 8}, Heldermann-Verlag, Berlin, 1990, 357-430.


\bibitem{loops} P. T. Nagy and K. Strambach, \textit{Loops in groups theory
and Lie theory}, de Gruyter Expositions in Mathematics. 35. Berlin-New York, 2002. 

\bibitem{oneill} B. O'Neill, \textit{Semi-Riemannian Geometry}, Pure and 
Applied Mathematics, Academic Press, New York-London, 1983. 

\bibitem{scheerer} H. Scheerer, \textit{Restklassenr\"aume kompakter 
zusammenh\"angender Gruppen mit Schnitt}, Math. Ann. {\bf 206} (1973), 
149-155. 


\bibitem{wolf} J. Wolf, \textit{The geometry and structure of isotropy 
irreducible homogeneous spaces}, Acta Math. {\bf 120} (1968), 59-148. 


\end{thebibliography}
\end{document}